\RequirePackage{etex}
\RequirePackage{easybmat}
\documentclass[]{article}
\usepackage{minitoc}
\usepackage[titletoc]{appendix}
\usepackage[utf8]{inputenc}
\usepackage[T1]{fontenc}
\usepackage{graphicx}
\usepackage{amsmath}
\usepackage{listings}
\usepackage{moreverb}
\usepackage{eurosym}
\usepackage{fancyvrb}
\usepackage{afterpage}
\usepackage[final]{pdfpages}
\usepackage[nottoc]{tocbibind}
\usepackage{listings}
\usepackage{amsthm}
\usepackage{amsfonts}
\usepackage{geometry}
\usepackage{xypic}
\xyoption{all}
\xyoption{arc}
\xyoption{poly}
\usepackage[]{algorithm2e}
\usepackage{hyperref}
\usepackage{fixltx2e}
\usepackage{amssymb}
\usepackage{tikz}
\usepackage{tkz-graph}
\usetikzlibrary{calc}
\usepackage{float}
\usepackage{color}
\usepackage{lipsum}
\usepackage{xypic}
\xyoption{all}
\xyoption{arc}
\xyoption{poly}

\newcommand*{\mathcolor}{}
\def\mathcolor#1#{\mathcoloraux{#1}}

\newtheorem{theorem}{Theorem}[section]
\newtheorem{lemma}[theorem]{Lemma}
\newtheorem{proposition}[theorem]{Proposition}
\newtheorem{cor}[theorem]{Corollary}
\newtheorem{definition}[theorem]{Definition}
\newtheorem{example}[theorem]{Example}

\newtheorem{remark}[theorem]{Remark}

\newtheorem{convention}[theorem]{Convention}

\newcommand\encircle[1]{%
  \tikz[baseline=(X.base)] 
    \node (X) [draw, shape=circle, inner sep=0] {\strut #1};}

\makeatletter
\def\v@rt#1#2{\m@th\ooalign{$\hfil#1|\hfil$\crcr$#1#2$}}
\def\captr{\mathrel{\mathpalette\v@rt\cap}}
\makeatother

\makeatletter

\newcommand{\Rmnum}[1]{\expandafter\@slowromancap\romannumeral #1@}
\makeatother

\newcommand{\drawEdge}[2]{\draw[-,thick] #1--#2;}
\newcommand{\drawDashEdge}[2]{\draw[-,dashed,thick] #1--#2;}
\newcommand{\drawArrow}[3]{
  \draw[-,thick] ($#1!.45!#2$)--($#1!.55!#2+#3$);
  \draw[-,thick] ($#1!.45!#2$)--($#1!.55!#2-#3$);
}
\newcommand{\drawTripleEdge}[3]{
  \draw[-,thick] #1--#2; 
  \draw[-,thick,yshift=.5mm] #1--#2; 
  \draw[-,thick,yshift=-.5mm] #1--#2;
  \drawArrow{#1}{#2}{#3}
}
\newcommand{\drawDoubleEdge}[3]{
  \draw[-,thick,yshift=.3mm] #1--#2; 
  \draw[-,thick,yshift=-.3mm] #1--#2;
  \drawArrow{#1}{#2}{#3}
}
\newcommand{\drawDashDoubleEdge}[3]{
  \draw[-,dashed,thick,yshift=.3mm] #1--#2; 
  \draw[-,dashed,thick,yshift=-.3mm] #1--#2;
  \drawArrow{#1}{#2}{#3}
}
\newcommand{\drawInfEdge}[2]{
  \drawDashEdge{#1}{#2} 
  \node[anchor=south] at ($#1!.5!#2$) {$\infty$};
}
\newcommand{\drawDotEdge}[2]{
  \fill ($#1!.3!#2$) circle (.4mm);
  \fill ($#1!.5!#2$) circle (.4mm);
  \fill ($#1!.7!#2$) circle (.4mm);
}
\newcommand{\drawRegDot}[1]{\fill #1 circle (.7mm);\fill[color=black!80] #1 circle (.5mm);}
\newcommand{\drawSpeDot}[1]{\fill #1 circle (1mm);\fill[color=white] #1 circle (.7mm);}
\newcommand{\drawESpDot}[1]{\fill #1 circle (1mm);\fill[color=black!80] #1 circle (.7mm);}
\newcommand{\wt}{\widetilde}

\newcommand\blfootnote[1]{%
  \begingroup
  \renewcommand\thefootnote{}\footnote{#1}%
  \addtocounter{footnote}{-1}%
  \endgroup
}

\begin{document}

\newgeometry{paperwidth=20cm,
left=2cm,
right=2cm,
paperheight=30cm,
top=2cm,
bottom=2.5cm}

\title{Interval Garside Structures Related to the Affine Artin Groups of Type $\widetilde{A}$}
\author{Georges Neaime\\
Universit\"at Bielefeld\\
\texttt{gneaime@math.uni-bielefeld.de}
}

\date{}

\maketitle

\begin{abstract}

Garside theory emerged from the study of Artin groups and their generalizations. Finite-type Artin groups admit two types of interval Garside structures corresponding to their standard and dual presentations. Concerning affine Artin groups, Digne established interval Garside structures for two families of these groups by using their dual presentations. Recently, McCammond established that none of the remaining dual presentations (except for one additional case) correspond to interval Garside structures. In this paper, shifting attention from dual presentations to other nice presentations for the affine Artin group of type $\tilde{A}$ discovered by Shi and Corran--Lee--Lee, I will construct interval Garside structures related to this group. This construction is the first successful attempt to establish interval Garside structures not related to the dual presentations in the case of affine Artin groups.

\end{abstract}

\begin{small}
\tableofcontents 
\end{small}

\blfootnote{This paper is written during my postdoc at Universität Bielefeld (DFG project: BA 2200/5-1). The Deutsche Forschungsgemeinschaft is acknowledged.}

\restoregeometry 

\section{Introduction}

In his PhD thesis \cite{GarsideThesis}, defended in 1965, and in the article that followed \cite{GarsideArt}, Garside solved the Word and Conjugacy Problems for the usual Artin's braid group by using a convenient monoid. In the beginning of the 1970's, it was realized by Brieskorn--Saito \cite{GarsideBrieskornSaito} and by Deligne \cite{GarsideDeligne} that Garside's approach extends to all Artin groups related to finite Coxeter groups. At the end of the 1990's, Dehornoy--Paris \cite{GarsideDehornoyParis} pursued and extended this approach, which leads to a larger family of groups now known as Garside groups (see \cite{GarsideBookPatrick}). A Garside group is realized as the group of fractions of a monoid, called a Garside monoid, in which there exist divisibility relations that provide relevant information about the Garside group. An important feature about Garside structures is the existence of normal forms that enable us to solve the Word and Conjugacy Problems for the associated groups. Garside structures are also desirable because they enjoy important group-theoretical, homological, and homotopical properties. It is therefore interesting to construct Garside structures for a given group. 

There is a way to construct Garside structures from intervals in a group (see \cite{GarsideBookPatrick}). The resulting Garside structures are called interval Garside structures. A given Garside group admits possibly several interval Garside structures. In the case of finite-type Artin groups (Artin groups related to finite Coxeter groups), there exist two interval Garside structures: the standard and dual interval Garside structures (developed in Subsection 2.4, see \cite{GarsideBookPatrick} for a detailed description). The standard structure involves the standard Coxeter generators, whereas the dual structure involves all the reflections inside the Coxeter group. Concerning affine Artin groups (Artin groups related to affine Coxeter groups that we recall in the next section), Digne \cite{Digne-Garside-Atilde, Digne-Garside-Ctilde} established interval Garside structures for two families of these groups (naimly, the groups of type $\tilde{A}$ and $\tilde{C}$) by using their dual presentations. McCammond \cite{McCammondFailureLattice} established that none of the remaining dual presentations (except for the group of type $\tilde{G}_2$) correspond to interval Garside structures. He and Sulway \cite{McCammond-Sulway} also identified these Artin groups as subgroups of other Garside groups, thereby clarifying some of their properties. In this paper, we construct a new interval Garside structure for the affine Artin group of type $\tilde{A}$ that does not involve its dual presentation. This provides a first successful attempt and an impetus to establish interval Garside structures (not related to the dual presentations) for the affine Artin groups.

Our construction is based on the discoveries of Shi \cite{ShiGenericVersion} and Corran--Lee--Lee \cite{CorranLeeLee} concerning the affine Artin group of type $\tilde{A}$: Shi showed that this group is isomorphic to the group denoted by $B(\infty,\infty,n)$ and Corran--Lee--Lee established a presentation for this group. We can think about $B(\infty,\infty,n)$ as a generic version of the complex braid group $B(e,e,n)$, where $e$ and $n$ are two parameters. The corresponding complex reflection group is denoted by $G(e,e,n)$. It is one of the infinite series of the classification of Shephard--Todd \cite{ShephardTodd} of the complex reflection groups. In a previous work \cite{GeorgesNeaimeIntervals}, I established interval Garside structures for the complex braid group $B(e,e,n)$ by constructing suitable lattices in $G(e,e,n)$. Using a generic version of the complex reflection group $G(e,e,n)$, also introduced by Shi in \cite{ShiGenericVersion} and denoted by $G(\infty,\infty,n)$, I am able, in this paper, to generalize my method to the groups $G(\infty,\infty,n)$ and $B(\infty,\infty,n)$ by using the presentations of Corran--Lee--Lee \cite{CorranLeeLee}. The method consists of constructing lattices of the balanced elements in $G(\infty,\infty,n)$ that are of maximal length, which gives rise to interval Garside structures for $B(\infty,\infty,n)$. Since $B(\infty,\infty,n)$ is isomorphic to the affine Artin group of type $\tilde{A}_{n-1}$ \cite{ShiGenericVersion}, our construction establishes interval Garside structures for the affine Artin groups of type $\tilde{A}$. Note that the generating set that we obtain is smaller than the dual generating set. In addition, our construction enables us to obtain Garside structures for the 3-parameter complex braid groups $B(de,e,n)$ related to all the infinite series in the classification of Shephard--Todd of the complex reflection groups.

The article is organized as follows. In Section 2, we define Garside structures and interval Garside structures. We briefly mention the various consequences of the existence of Garside structures in order to show the richness of the subject. We also define Artin groups and some generalizations. Then, we develop two fundamental constructions related to finite-type Artin groups, namely the standard and dual Garside constructions. Section 3 paves the way for our proof by recalling the presentations of Shi and Corran--Lee--Lee for the affine Artin groups of type $\tilde{A}$. We also mention the connection with the complex braid groups of type $(e,e,n)$ and recall the dual interval Garside structures that exist for some cases of the affine Artin groups. In Section 4, we establish the construction of the interval Garside strcutures. The proof is an adaptation of our constructions in \cite{GeorgesNeaimeIntervals} to type $(\infty,\infty,n)$. In the last section, we identify the Artin groups of type $\tilde{A}$ and provide a description of all the other Garside structures that naturally appear in our construction.

\section{Definitions and preliminaries} 

\subsection{Garside monoids and groups}

Let $M$ be a monoid. For $f, g$ in $M$, say that $f$ left-divides $g$, written  $f \preceq g$, if $fg' = g$ holds for some $g'$ lying in $M$. Symmetrically, say that $f$ right-divides $g$, written $f \preceq_r g$, if $g'f=g$ holds for some $g' \in M$.

\begin{definition}\label{DefGarsideMonoid}

A Garside monoid is a pair $(M, \Delta)$, where $M$ is a monoid satisfying

\begin{enumerate}

\item there exists $\lambda : M \longrightarrow \mathbb{N}$ such that $\lambda(fg) \geq \lambda(f) + \lambda(g)$ and $g \neq 1 \Longrightarrow \lambda(g) \neq 0$,

\item $M$ is cancellative, that is, $hfk = hgk \Longrightarrow f=g$ for $f,g,h,k \in M$,

\item any two elements of $M$ have a gcd and an lcm for $\preceq$ and $\preceq_r$,

\item $\Delta$ is a \emph{Garside element} of $M$, this meaning that the set of its left divisors coincides with the set of its right divisors and generate $M$.

\end{enumerate} 

The divisors of $\Delta$ are called the simples of $M$.

\end{definition}


It is well-known result due to Ore \cite{OreConditions} that a cancellative monoid in which every pair of elements has common multiples embeds into a well-defined object called its group of fractions. Assumptions 2 and 3 of Definition \ref{DefGarsideMonoid} ensure that Ore’s conditions are satisfied. This allows us to give the following definition.

\begin{definition}\label{DefGarsideGroups}

A Garside group is the group of fractions of a Garside monoid.

\end{definition}

The pair $(M,\Delta)$ and the group of fractions of the monoid $M$ is called a Garside structure for $M$.

\begin{remark}

Note that the set of simples associated with Garside groups have traditionally been required to be finite (see \cite{GarsideBookPatrick}). This requirement has not been incorporated into the definition given above. 

\end{remark}

\subsection{Interval Garside structures}\label{SubsectionIntervalGarsideStructures}

There is a way to construct Garside structures from intervals in a given group. Let $W$ be a group positively generated by a set $S$ ($S \subset W$ generates $W$ as a monoid). Note that we do not require the set $S$ to be finite.

Let $\mathbf{S}$ be a set in bijection with $S$. Denote by $\boldsymbol{\ell}(\boldsymbol{w})$ the word length over $\mathbf{S}$ of the word $\boldsymbol{w} \in \mathbf{S}^{*}$. Let us start by the following definition.

\begin{definition}\label{def.Slength.reducedword}
Let $w \in W$. We define $\ell(w)$ to be the minimal word length $\boldsymbol{\ell}(\boldsymbol{w})$ of a word $\boldsymbol{w}$ over $\mathbf{S}$ that represents $w$. A reduced expression of $w$ is any word representative of $w$ of word length $\ell(w)$.
\end{definition}

Let us define a partial order relation on $W$ as follows.

\begin{definition}\label{DefPartalOrder11}

Let $v, w \in W$. We say that $v$ is a divisor of $w$ or $w$ is a multiple of $v$, and write $v \preceq w$, if $w = v u$ with $u \in W$ and $\ell(w) = \ell(v) + \ell(u)$, where $\ell(w)$ is the length over $S$ of $w \in W$.

\end{definition}

\begin{definition}\label{DefGeneralIntervalMonoid}

For $w \in W$, define a monoid $M([1,w]^S)$ by the monoid presentation:

\begin{itemize}

\item generating set $\underline{P}$ in bijection with the interval
\begin{center} $[1,w]^S := \{ v \in W \ | \ 1 \preceq v \preceq w \}$, \end{center}

\item relations: $\underline{v}\ \underline{u} = \underline{t}$ if $v, u, t \in [1,w]$, $vu=t$, and $v \preceq t$, that is, \mbox{$\ell(v) + \ell(u) = \ell(t)$}.

\end{itemize}

\end{definition}

Symmetrically, one can define the partial order relation on $W$ as follows: 
\begin{center}$v \preceq_r w$ if and only if $\ell(w{v}^{-1}) + \ell(v) = \ell(w)$, \end{center}
then define the interval $[1,w]^S_r$. 

\begin{definition}

Let $w$ be in $W$. We say that $w$ is a balanced element of $W$ if $[1,w]^S = [1,w]^S_r$ holds.

\end{definition}

The following result is stated in \cite{Digne-Garside-Atilde} (Th\'eor\`eme 5.4).

\begin{theorem}\label{TheoremMichelGarside}

If $w \in W$ is balanced and both posets $([1,w]^S,\preceq)$ and $([1,w]^S_r, \preceq_r)$ are lattices, then $(M([1,w]^S),\underline{w})$ is a Garside monoid with simples $\underline{[1,w]^S}$.

\end{theorem}

Under the hypothesis of Theorem \ref{TheoremMichelGarside}, the monoid $M([1,w]^S)$ is called an interval Garside monoid. Since it is a Garside monoid, its group of fractions exists and is denoted by $G([1,w]^S)$. We call it an interval Garside group. An interval Garside structure consists of an interval Garside monoid and its group of fractions. 

\begin{remark}

The reader should note that we are using ``interval Garside structure'' in the expanded sense of Digne (\cite{Digne-Garside-Atilde,Digne-Garside-Ctilde}) as we are not requiring the simples to be finite.

\end{remark}

The next definition describes the Garside normal forms (see \cite{GarsideBookPatrick}).

\begin{definition}\label{NormalForms}

Given a Garside monoid $(M,\Delta)$, elements in $M$ have normal forms: any element can be uniquely written $x_1x_2\cdots x_k$ with $x_i$ a divisor of $\Delta$ such that $x_i$ is equal to the left greatest common divisor of $x_ix_{i+1} \cdots x_k$ and $\Delta$. In the Garside group, any element can be written as $\Delta^k x$ with $x \in M$ and $k\in \mathbb{Z}$.

\end{definition}

The next proposition provides some consequences of being an interval Garside group (see Theorem 2.12 in \cite{McCammond-Sulway}). Recall that we are not requiring the generating set $S$ or the set of simples of the Garside structure to be finite.

\begin{proposition}\label{PropConseqGarside}

Let $G$ be a Garside group. We have:

\begin{itemize}

\item The Word Problem for $G$ is solvable.
\item The group $G$ admits a finite dimentional Eilenberg-MacLane space $\Sigma$ of type $K(G,1)$. It follows that the cohomology of $G$ is equal to that of $\Sigma$, the cohomological dimension of $G$ is bounded above by the dimension of $\Sigma$, which turns out to be equal to the height of the lattice of simples. It also follows that the group $G$ is torsion-free.

\end{itemize}

\end{proposition}

\subsection{Principal objects of study}

A complex reflection is a linear transformation of finite order, which fixes a hyperplane pointwise. Let $W$ be a finite subgroup of $GL_n(\mathbb{C})$ with $n \geq 1$ and $\mathcal{R}$ be the set of complex reflections of $W$. We say that $W$ is a complex reflection group if $W$ is generated by $\mathcal{R}$. It is well known that every complex reflection group is a direct product of irreducible ones. These irreducible complex reflection groups have been classified by Shephard and Todd \cite{ShephardTodd} in 1954. The classification consists of the following cases:
\begin{itemize}
\item the infinite series $G(de,e,n)$ depending on three positive integer parameters,
\item $34$ exceptional groups.
\end{itemize}

As we are interested in the groups of the infinite series, we provide the definition of the group $G(de,e,n)$. For the definition of the $34$ exceptional groups, see \cite{ShephardTodd}.
\begin{definition}\label{DefinitionG(de,e,n)}
The group $G(de,e,n)$ is defined as the group of $n \times n$ monomial matrices (each row and column has a unique nonzero entry), where
\begin{itemize}
\item all nonzero entries are $de$-th roots of unity and
\item the product of all the nonzero entries is a $d$-th root of unity.
\end{itemize} 
\end{definition}

Brou\'e, Malle, and Rouquier \cite{BMR} managed to associate a complex braid group $B$ to each complex reflection group $W$. The definition of the complex braid group related to $W$ is as follows. Let $\mathcal{A} := \{\ker(s-1)\ \vert\ s \in \mathcal{R} \}$ be the hyperplane arrangement and $X := \mathbb{C}^n\setminus\bigcup\mathcal{A}$ be the hyperplane complement. The complex reflection group $W$ acts naturally on $X$. Let $X/W$ be its space of orbits. 
\begin{definition}
The complex braid group associated with $W$ is defined as the fundamental group:  $$B:=\pi_1(X/W).$$
\end{definition}

If $W$ is equal to $G(de,e,n)$, then we denote by $B(de,e,n)$ the associated complex braid group.\\

Let $W$ be a real reflection group, meaning that $W$ is a subgroup of $GL_n(\mathbb{R})$. By \cite{Bourbaki} and \cite{BrieskornArtinTitsFundamentalGroup}, we recover in the previous definitions the notion of finite-type Coxeter and Artin groups. Let us recall the general definition of Coxeter and Artin groups.

\begin{definition}\label{DefinitionCoxeterGroups}

Assume that $W$ is a group and $S$ is a subset of $W$. For $s$ and $t$ in $S$, let $m_{st}$ be the order of $st$ if this order is finite, and be $\infty$ otherwise. We say that $(W,S)$ is a Coxeter system, and that $W$ is a Coxeter group, if $W$ admits the presentation with generating set $S$ and relations: 
\begin{itemize}
\item quadratic relations: $s^2=1$ for all $s \in S$ and
\item braid relations: $\underset{m_{st}}{\underbrace{sts\cdots}}=\underset{m_{st}}{\underbrace{tst\cdots}}$  for $s,t \in S$, $s \neq t$ and $m_{st} \neq \infty$.
\end{itemize}
\end{definition}

We define the Artin group $B(W)$ associated with a Coxeter system $(W,S)$ as follows.

\begin{definition}\label{DefinitionArtinTits}
The Artin group $B(W)$ associated with a Coxeter system $(W,S)$ is defined by a presentation with generating set $\widetilde{S}$ in bijection with the generating set $S$ of the Coxeter group and the relations are only the braid relations: $\underset{m_{st}}{\underbrace{\tilde{s}\tilde{t}\tilde{s}\cdots}}=\underset{m_{st}}{\underbrace{\tilde{t}\tilde{s}\tilde{t}\cdots}}$ for $\tilde{s}, \tilde{t} \in \widetilde{S}$ and $\tilde{s} \neq \tilde{t}$, where $m_{st} \in \mathbb{Z}_{\geq 2}$ is the order of $st$ in $W$.
\end{definition}

The classification of finite Coxeter groups is classical and is essentially due to Coxeter (see \cite{Bourbaki}). Finite-type Artin groups are Artin groups related to finite Coxeter groups. All of these groups have presentations that are encoded in the well-known Coxeter diagrams.\\

Let $M(W)$ denotes the real symmetric matrix $$(\hbox{cos}(\pi - \pi/m_{st}))_{s,t \in S}$$ called the cosine matrix associated with $W$. Affine Coxeter groups are those for which the cosine matrix is positive semi-definite with one zero eigenvalue. Recall that the structure of these groups play a central role in the classification of complex semi-simple Lie algebras. We call affine Artin groups the Artin groups related to affine Coxeter groups.\\

The presentations of affine Coxeter and Artin groups are encoded in the well-known Coxeter diagrams shown in Figure \ref{fig:dynkin}. The four infinite families are denoted by $\tilde{A}_n$, $\tilde{B}_n$, $\tilde{C}_n$, $\tilde{D}_n$, and the five exceptional cases by $\tilde{E}_6$, $\tilde{E}_7$, $\tilde{E}_8$, $\tilde{F}_4$, $\tilde{G}_2$. Removing the white vertex and the attached dashed edges from each diagram produces the diagram presentation of a finite Coxeter group and its related Artin group. For example, if we remove the white dot from $\tilde{A}_n$ ($n \geq 2$) and the related edges, we recover the $A_n$ diagram that encodes a presentation of the symmetric group $\mathfrak S_n$ and the related Artin group is the usual braid group $\mathcal{B}_n$ that is defined by the following presentation.

\begin{definition}\label{DefinitionClasBraidGroup}
The usual braid group $\mathcal{B}_n$ is defined by a presentation with generators $\tilde{s}_1, \tilde{s}_2, \cdots, \tilde{s}_{n-1}$ and relations:
\begin{enumerate}
\item $\tilde{s}_{i}\tilde{s}_{i+1}\tilde{s}_i = \tilde{s}_{i+1}\tilde{s}_i\tilde{s}_{i+1}$ for $1 \leq i \leq n-2$,
\item $\tilde{s}_i\tilde{s}_j=\tilde{s}_j\tilde{s}_i$ for $|i-j| > 1$.
\end{enumerate}
\end{definition}

Sometimes we denote the finite-type and affine Artin groups by $B(X_n)$ and $B(\tilde{X}_n)$, respectively. We use this notation instead of $B(W)$ to indicate the type of the corresponding Coxeter diagram. 

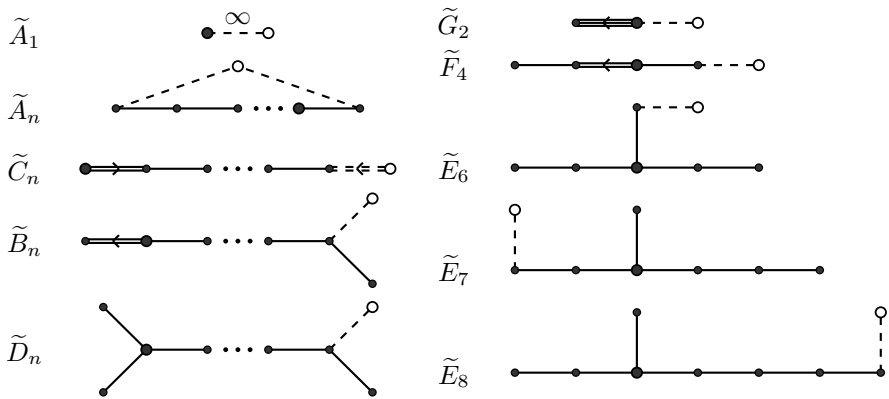
\begin{figure}[]
\begin{center}
  \begin{tabular}{cc}
    \begin{tikzpicture}[scale=.8]
    \begin{scope}[yshift=5.25cm,xshift=2cm]
      \node at (-3,0) {$\wt A_1$};
      \drawInfEdge{(0,0)}{(1,0)}
      \drawSpeDot{(1,0)}
      \drawESpDot{(0,0)}
    \end{scope}
    \begin{scope}[yshift=4cm,xshift=.5cm]
      \node at (-1.5,0) {$\wt A_n$};
      \drawEdge{(0,0)}{(2,0)}
      \drawDashEdge{(0,0)}{(2,.7)}
      \drawDashEdge{(4,0)}{(2,.7)}
      \drawEdge{(3,0)}{(4,0)}
      \drawDotEdge{(2,0)}{(3,0)}
      \foreach \x in {0,1,2,4} {\drawRegDot{(\x,0)}}
      \drawSpeDot{(2,.7)}
      \drawESpDot{(3,0)}
    \end{scope}
    \begin{scope}[yshift=3cm]
      \node at (-1,0) {$\wt C_n$};
      \drawDoubleEdge{(1,0)}{(0,0)}{(0,.1)}
      \drawEdge{(1,0)}{(2,0)}
      \drawDotEdge{(2,0)}{(3,0)}
      \drawEdge{(3,0)}{(4,0)}
      \drawDashDoubleEdge{(4,0)}{(5,0)}{(0,.1)}
      \foreach \x in {1,2,3,4} {\drawRegDot{(\x,0)}}
      \drawESpDot{(0,0)}
      \drawSpeDot{(5,0)}
    \end{scope}
    \begin{scope}[yshift=1.8cm]
      \node at (-1,0) {$\wt B_n$};
      \drawDoubleEdge{(0,0)}{(1,0)}{(0,.1)}
      \drawEdge{(1,0)}{(2,0)}
      \drawDotEdge{(2,0)}{(3,0)}
      \drawEdge{(3,0)}{(4,0)}
      \drawDashEdge{(4,0)}{(4.707,.707)}
      \drawEdge{(4,0)}{(4.707,-.707)}
      \drawRegDot{(4.707,-.707)}
      \foreach \x in {0,2,3,4} {\drawRegDot{(\x,0)}}
      \drawESpDot{(1,0)}
      \drawSpeDot{(4.707,.707)}
    \end{scope}
    \begin{scope}
      \node at (-1,0) {$\wt D_n$};
      \drawEdge{(.293,.707)}{(1,0)}
      \drawEdge{(.293,-.707)}{(1,0)}
      \drawEdge{(1,0)}{(2,0)}
      \drawDotEdge{(2,0)}{(3,0)}
      \drawEdge{(3,0)}{(4,0)}
      \drawDashEdge{(4,0)}{(4.707,.707)}
      \drawEdge{(4,0)}{(4.707,-.707)}
      \drawRegDot{(.293,.707)}
      \drawRegDot{(.293,-.707)}
      \drawRegDot{(4.707,-.707)}
      \foreach \x in {2,3,4} {\drawRegDot{(\x,0)}}
      \drawESpDot{(1,0)}
      \drawSpeDot{(4.707,.707)}
    \end{scope}
  \end{tikzpicture}
  &
  \begin{tikzpicture}[scale=.8]
  \begin{scope}[yshift=5.8cm,xshift=1cm]
    \node at (-2,0) {$\wt G_2$};
    \drawTripleEdge{(0,0)}{(1,0)}{(0,.1)}
    \drawDashEdge{(1,0)}{(2,0)}
    \drawRegDot{(0,0)}
    \drawESpDot{(1,0)}
    \drawSpeDot{(2,0)}
  \end{scope}
  \begin{scope}[yshift=5.1cm]
    \node at (-1,0) {$\wt F_4$};
    \drawDoubleEdge{(1,0)}{(2,0)}{(0,.1)}
    \drawEdge{(0,0)}{(1,0)}
    \drawEdge{(2,0)}{(3,0)}
    \drawDashEdge{(3,0)}{(4,0)}
    \foreach \x in {0,1,3} {\drawRegDot{(\x,0)}}
    \drawESpDot{(2,0)}
    \drawSpeDot{(4,0)}
  \end{scope}
  \begin{scope}[yshift=3.4cm]
    \node at (-1,0) {$\wt E_6$};
    \drawEdge{(0,0)}{(4,0)}
    \drawEdge{(2,0)}{(2,1)}
    \drawDashEdge{(2,1)}{(3,1)}
    \foreach \x in {0,1,3,4} {\drawRegDot{(\x,0)}}
    \drawRegDot{(2,1)}
    \drawESpDot{(2,0)}
    \drawSpeDot{(3,1)}
  \end{scope}
  \begin{scope}[yshift=1.7cm]
    \node at (-1,0) {$\wt E_7$};
    \drawEdge{(0,0)}{(5,0)}
    \drawEdge{(2,0)}{(2,1)}
    \drawDashEdge{(0,0)}{(0,1)}
    \foreach \x in {0,1,3,4,5} {\drawRegDot{(\x,0)}}
    \drawRegDot{(2,1)}
    \drawESpDot{(2,0)}
    \drawSpeDot{(0,1)}
  \end{scope}
  \begin{scope}
    \node at (-1,0) {$\wt E_8$};
    \drawEdge{(0,0)}{(6,0)}
    \drawDashEdge{(6,0)}{(6,1)}
    \drawEdge{(2,0)}{(2,1)}
    \foreach \x in {0,1,3,4,5,6} {\drawRegDot{(\x,0)}}
    \drawRegDot{(2,1)}
    \drawESpDot{(2,0)}
    \drawSpeDot{(6,1)}
  \end{scope}
  \end{tikzpicture}
  \end{tabular}
  \caption{Four infinite families and five exceptional cases.\label{fig:dynkin}}
  \end{center}
\end{figure}

\subsection{Two fundamental constructions}

Let $(W,S)$ be a finite Coxeter system with $|S| = n$. Let $T$ be the closure of $S$ under conjugation in $W$. There are certain elements in $W$ that play a special role in the construction of interval Garside structures for the associated Artin group $B(W)$. There always exists a unique element $w_0$ in $W$ which is of maximal length over $S$. It is called the longest element in $W$. There are many elements that are of maximal length over $T$. Some of these elements can be found by taking the product of the elements in $S$ in any order. These particular elements are called the Coxeter elements of $W$. We provide the following two seminal examples of interval Garside structures for $B(W)$ constructed from intervals in the finite Coxeter groups. We apply the construction explained in Subsection \ref{SubsectionIntervalGarsideStructures}. The first example is related to the longest element and the next one is related to Coxeter elements (check \cite{GarsideBookPatrick} for more details).\\

\textbf{Standard construction.} We have that the interval of the divisors of the unique longest element $w_0 \in W$ is equal to the Coxeter group $W$. Construct the interval monoid $M([1,w_0]^S)$ as in Definition \ref{DefGeneralIntervalMonoid}. We have that $M([1,w_0]^S)$ is isomorphic to the Artin monoid $B^{+}(W)$, that is, the monoid defined by the same presentation as $B(W)$ (see Definition \ref{DefinitionArtinTits}). Hence $B^{+}(W)$ is generated by a copy $\underline{W}$ of $W$ with $\underline{f}\ \underline{g} = \underline{h}$ if $fg=h$ and $\ell_S(f) + \ell_S(g) = \ell_S(h)$ for $f,g$, and $h \in W$. It is also known that $w_0$ is balanced and both posets $([1,w_0]^S,\preceq)$ and $([1,w_0]^S_r, \preceq_r)$ are lattices. Hence by \mbox{Theorem \ref{TheoremMichelGarside},} we have the following result.

\begin{theorem}\label{TheoremGarsideStandard}

Let $(W,S)$ be a finite Coxeter system and $B(W)$ the Artin group associated with $W$. The monoid $(B^{+}(W),\underline{w_0})$ is an interval Garside monoid with simples $\underline{W}$, where $w_0$ is the longest element for the generating set $S$ and $\underline{W}$ is a copy of the Coxeter group.  The group of fractions $G([1,w_0]^S)$ of the monoid $M([1,w_0]^S)$ exists and is isomorphic to the corresponding Artin group $B(W)$. This interval Garside structure is usually called the standard Garside structure.

\end{theorem}

\begin{remark}

We call this construction standard since it has been used by Garside himself in his PhD thesis \cite{GarsideThesis} where he solved the Conjugacy Problem for the usual braid groups. His construction was generalized by Brieskorn--Saito \cite{GarsideBrieskornSaito} and by Deligne \cite{GarsideDeligne} to all finite-type Artin groups.

\end{remark}

\textbf{Dual construction.} Let $c$ be any Coxeter element in $W$. We consider this time the length of the elements of $W$ over $T$. Construct the interval $[1,c]^T$ and the associated interval monoid $M([1,c]^T)$ as in Definition \ref{DefGeneralIntervalMonoid}. This interval monoid is isomorphic to the dual Artin monoid defined by Bessis in \cite{BessisDualBraidENS}. The element $c$ is balanced (since the length (over $T$) is invariant by conjugacy). Moreover, the interval $[1,c]^T$ is the interval of the generalized noncrossing partitions denoted by NC$(W,c)$ (see \cite{BessisDualBraidENS}). It is a lattice by a result of Brady-Watt \cite{BradyWattLattices} and Bessis \cite{BessisDualBraidENS}. Hence by \mbox{Theorem \ref{TheoremMichelGarside},} we have the following result.

\begin{theorem}\label{TheoremGarsideDual}

Let $(W,S)$ be a finite Coxeter system and let $T$ be the closure of $S$ under conjugation in $W$. Let $c$ be any Coxeter element in $W$. The monoid $(M([1,c]^T),\underline{c})$ is an interval Garside monoid with simples $\underline{[1,c]^T}$, where $c$ is a Coxeter element and $[1,c]^T$ is the lattice of the generalized noncrossing partitions. The group of fractions $G([1,c]^T)$ of the monoid $M([1,c]^T)$ exists and is isomorphic to the corresponding Artin group. We say that that the interval $[1,c]^T$ encodes a dual presentation for the Artin group. This interval Garside structure is usually called the dual Garside structure.

\end{theorem}

\section{The affine and complex cases}

Consider now a general Coxeter group $W$ (not necessarily finite) generated by $S$. It does not make sense to define an element $w_0$, as “longest element” no longer exists in the general case. However, it still makes sense to define a Coxeter element $c$ as the product of the elements in $S$ in some order and to consider the interval $[1,c]^T$ and the associated interval monoid $M([1,c]^T)$. The element $c$ will always be balanced as we consider the length over $T$ that is invariant by conjugacy. It follows that the only obstruction in constructing Garside monoids and groups is to show that the interval $[1,c]^T$ is a lattice. When it is a lattice, we obtain an interval Garside structure by applying Theorem \ref{TheoremMichelGarside} and the group of fractions of $M([1,c]^T)$ exists. If it is isomorphic to the corresponding Artin group, then we also say that that the interval $[1,c]^T$ encodes a dual presentation for the Artin group. We also call this Garside structure the dual interval Garside structure. It turns out that the lattice property is difficult to prove in general. Let us discuss the case of affine Coxeter groups.

\subsection{The affine case}

In the case of affine Coxeter groups, Digne constructed lattices associated to some carefully choosen Coxeter elements that gave rise to dual interval Garside structures for the affine Artin groups of type $\tilde{A}_n$ and $\tilde{C}_n$ (see \cite{Digne-Garside-Atilde, Digne-Garside-Ctilde}). McCammond investigated the lattice property for the other cases of affine Artin groups (see \cite{McCammondFailureLattice}). We therefore have the following result.

\begin{theorem}\label{TheoremGarsideEuclidean}

There exists dual interval Garside structures for the affine Artin groups of type $\tilde{A}_n$, $\tilde{C}_n$, and $\tilde{G}_2$. For all the other affine Artin groups, the intervals corresponding to Coxeter elements are not lattices and thus do not give rise to dual interval Garside structures.

\end{theorem}

Although the previous result indicates the failure of the lattice property and thereby the failure of the existence of dual interval Garside structures, McCammond and Sulway \cite{McCammond-Sulway} succeeded to better understand the affine Artin groups by identifying them as subgroups of other new Garside groups. This result enables them to derive structural consequences for affine Artin groups. Actually, they showed that every irreducible affine Artin group is a torsion-free centerless group with a solvable word problem and a finite-dimensional classifying space (see Theorem D in \cite{McCammond-Sulway}). However, the question whether all affine Artin groups admit Garside structures remains open.

\subsection{Type $(e,e,n)$}

\subsubsection{Presentations of Brou\'e--Malle--Rouquier}

Recall that for $e,n \geq 1$, $G(e,e,n)$ is the group of $n \times n$ monomial matrices, with all nonzero entries lying in $\mu_{e}$, the $e$-th roots of unity, and for which the product of the nonzero entries is $1$. Broué, Malle, and Rouquier \cite{BMR} described a presentation of the complex braid group $B(e,e,n)$ attached to $G(e,e,n)$ by generators and relations as follows.

\begin{proposition}\label{PropPresBMGG(een)}
The complex braid group $B(e,e,n)$ is isomorphic to the group defined by a presentation with generators $\tilde{t}_0$, $\tilde{t}_1$, $\tilde{s}_3$, $\tilde{s}_4$, $\cdots$, $\tilde{s}_n$ and relations:
\begin{enumerate}

\item The braid relations for $\tilde{s}_3$, $\tilde{s}_4$, $\cdots$, $\tilde{s}_{n-1}$ given earlier in Definition \ref{DefinitionClasBraidGroup},
\item $\tilde{s}_3 \tilde{t}_i \tilde{s}_3 = \tilde{t}_i \tilde{s}_3 \tilde{t}_i$ for $i=0,1$,
\item $\tilde{s}_j \tilde{t}_i = \tilde{t}_i \tilde{s}_j$ for $i=0,1$ and $4 \leq j \leq n$,
\item $\tilde{s}_3 \tilde{t}_1\tilde{t}_0\tilde{s}_3\tilde{t}_1\tilde{t}_0 = \tilde{t}_1 \tilde{t}_0\tilde{s}_3\tilde{t}_1\tilde{t}_0\tilde{s}_3$,
\item $\underset{e}{\underbrace{\tilde{t}_0\tilde{t}_1\tilde{t}_0\cdots}} = \underset{e}{\underbrace{\tilde{t}_1\tilde{t}_0\tilde{t}_1\cdots}}$.

\end{enumerate}

\end{proposition}

This presentation can be described by the following diagram. It is called the presentation of BMR (Brou\'e--Malle--Rouquier) of $B(e,e,n)$

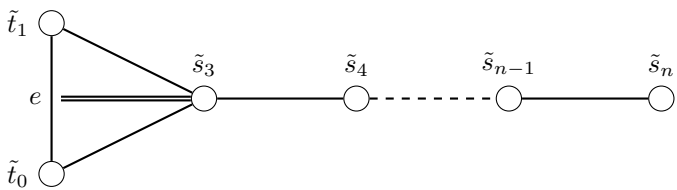
\begin{figure}[H]\label{PresofBMRBeen}
\begin{center}
\begin{tikzpicture}

\node[draw, shape=circle, label=left:$\tilde{t}_0$] (1) at (0,0) {};
\node[draw, shape=circle, label=left:$\tilde{t}_1$] (2) at (0,2) {};
\node[draw, shape=circle,label=above:$\tilde{s}_3$] (3) at (2,1) {};
\node[draw, shape=circle,label=above:$\tilde{s}_4$] (4) at (4,1) {};
\node[draw, shape=circle,label=above:$\tilde{s}_{n-1}$] (n-1) at (6,1) {};
\node[draw,shape=circle,label=above:$\tilde{s}_n$] (n) at (8,1) {};

\node[] (0) at (0,1) {};

\draw[thick,-] (1) to node[auto] {$e$} (2);
\draw[thick,-] (1) to (3);
\draw[thick,-] (2) to (3);
\draw[thick,-] (3) to (4);
\draw[thick,dashed,-] (4) to (n-1);
\draw[thick,-] (n-1) to (n);
\draw[double,thick,-] (3) to (0);

\end{tikzpicture}
\end{center}

\caption{Diagram for the presentation of BMR of $B(e,e,n)$.}
\end{figure}

It is also shown in \cite{BMR} that if one adds the quadratic relations for all the generators of the presentation of $B(e,e,n)$, one obtains a presentation of a group isomorphic to $G(e,e,n)$. It is called the presentation of BMR of $G(e,e,n)$. Let $\mathbf{t}_0$, $\mathbf{t}_1$, $\mathbf{s}_3$, $\mathbf{s}_4$, $\cdots$, and $\mathbf{s}_n$ be the generators of the presentation of BMR of $G(e,e,n)$. The matrices in $G(e,e,n)$ that correspond to the generators are given by $\mathbf{t}_i \longmapsto t_i:= \begin{pmatrix}

0 & \zeta_{e}^{-i} & 0\\
\zeta_{e}^{i} & 0 & 0\\
0 & 0 & I_{n-2}\\

\end{pmatrix}$ for $i=0,1$ and $\mathbf{s}_j \longmapsto s_j:= \begin{pmatrix}

I_{j-2} & 0 & 0 & 0\\
0 & 0 & 1 & 0\\
0 & 1 & 0 & 0\\
0 & 0 & 0 & I_{n-j}\\

\end{pmatrix}$ for $3 \leq j \leq n$, where $\zeta_e$ is the $e$-th root of unity that is equal to $exp(2i\pi/e)$ and $I_k$ is the rank $k$ identity matrix for $1 \leq k \leq n$.

\subsubsection{Presentations of Corran--Picantin}\label{SubsectionOtherPresentations}

Consider the presentation of BMR of the complex braid group $B(e,e,n)$. For $e \geq 3$ and $n \geq 3$, it is shown in \cite{CorranPhD} that the monoid defined by the corresponding monoid presentation fails to embed into $B(e,e,n)$. Thus, this presentation does not give rise to a Garside structure for $B(e,e,n)$. In \cite{CorranPicantin}, Corran and Picantin described another presentation for $B(e,e,n)$ and showed that it gives rise to a Garside structure for $B(e,e,n)$. This presentation consists in attaching a dual presentation of the dihedral group $I_2(e)$ to standard presentations of type $A$.

\begin{proposition}\label{PropositionPresCPB(een)}
The complex braid group $B(e,e,n)$ is isomorphic to a group defined by a presentation with generators $\tilde{t}_i$ ($i \in \mathbb{Z}/e \mathbb{Z}$), $\tilde{s}_3$, $\tilde{s}_4$, $\cdots$, $\tilde{s}_n$ and relations as follows.
\begin{enumerate}

\item the braid relations for $\tilde{s}_3$, $\tilde{s}_4$, $\cdots$, $\tilde{s}_{n-1}$ given earlier in Definition \ref{DefinitionClasBraidGroup},
\item $\tilde{s}_3 \tilde{t}_i \tilde{s}_3 = \tilde{t}_i \tilde{s}_3 \tilde{t}_i$ for $i \in \mathbb{Z}$,
\item $\tilde{s}_j \tilde{t}_i = \tilde{t}_i \tilde{s}_j$ for $i \in \mathbb{Z}$ and $4 \leq j \leq n$,
\item $\tilde{t}_i \tilde{t}_{i-1} = \tilde{t}_j \tilde{t}_{j-1}$ for $i,j \in \mathbb{Z}/e\mathbb{Z}$.

\end{enumerate}

\end{proposition}

This presentation is called the presentation of Corran--Picantin of $B(e,e,n)$. It can be described by the following diagram. The circle describes Relations 4 of Proposition~\ref{PropositionPresCPB(een)} (this corresponds to the dual presentation of the dihedral group $I_2(e)$). The other edges used to describe all the other relations follow the standard conventions for the usual braid group diagram of type $A$.

\begin{figure}[H]
\begin{center}
$$\begin{xy}
(-1.7, 0) *{\ellipse(9,15){}};
(4.2, 12) *++={\rule{0pt}{4pt}} *\frm{o};
    (20,0) *++={\rule{0pt}{4pt}} *\frm{o} **@{-};
(4.2,-12) *++={\rule{0pt}{4pt}} *\frm{o};
    (20,0) *++={\rule{0pt}{4pt}} *\frm{o} **@{-};
(7.1, 4) *++={\rule{0pt}{4pt}} *\frm{o};
    (20,0) *++={\rule{0pt}{4pt}} *\frm{o} **@{-};
(7.1,-4) *++={\rule{0pt}{4pt}} *\frm{o};
    (20,0) *++={\rule{0pt}{4pt}} *\frm{o} **@{-};
(30, 0) *++={\rule{0pt}{4pt}} *\frm{o} **@{-};
(40, 0) *++={\rule{0pt}{4pt}} *\frm{o} **@{-};
(50, 0) *++={\dots}  **@{-};
(60, 0) *++={\rule{0pt}{4pt}} *\frm{o} **@{-};
(9, 12) *++={\tilde{t}_1};
(10,-15) *++={\tilde{t}_{e-2}};
(10.5, 1) *++={\tilde{t}_0};
(11,-7.5) *++={\tilde{t}_{ _{e-1}}};
(23,-3) *++={\tilde{s}_3};
(33,-3) *++={\tilde{s}_4};
(43,-3) *++={\tilde{s}_5};
(63,-3) *++={\tilde{s}_n};
(0, 17.5) *++={\cdot};
(1.1, 17) *++={\cdot};
(2, 16) *++={\cdot};
(0, -17.5) *++={\cdot};
(1.1, -17) *++={\cdot};
(2, -16) *++={\cdot};
\end{xy}
$$
\end{center}
\caption{Diagram for the presentation of Corran--Picantin of $B(e,e,n)$.}\label{PresofCPBeen}
\end{figure}
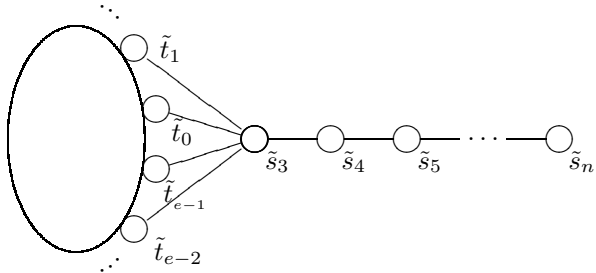

It is also shown in \cite{CorranPicantin} that if one adds the quadratic relations for all the generators of the presentation of Corran--Picantin of $B(e,e,n)$, one obtains a presentation of a group isomorphic to $G(e,e,n)$. It is called the presentation of Corran--Picantin of $G(e,e,n)$. Let $\mathbf{t}_0, \mathbf{t}_1, \cdots, \mathbf{t}_{e-1}, \mathbf{s}_3, \cdots, \mathbf{s}_n$ be the generators of the presentation of Corran--Picantin of $G(e,e,n)$. The matrices in $G(e,e,n)$ that correspond to the generators are given by $\mathbf{t}_i \longmapsto t_i:= \begin{pmatrix}

0 & \zeta_{e}^{-i} & 0\\
\zeta_{e}^{i} & 0 & 0\\
0 & 0 & I_{n-2}\\

\end{pmatrix}$ \mbox{for $0 \leq i \leq e-1$} and $\mathbf{s}_j \longmapsto s_j:= \begin{pmatrix}

I_{j-2} & 0 & 0 & 0\\
0 & 0 & 1 & 0\\
0 & 1 & 0 & 0\\
0 & 0 & 0 & I_{n-j}\\

\end{pmatrix}$ for $3 \leq j \leq n$.\\

In \cite{GeorgesNeaimeIntervals}, I constructed intervals in $G(e,e,n)$ that correspond to balanced elements that are of maximal length over the generating set of Corran--Picantin. These intervals are also proven to be lattices and this gives rise to interval Garside groups. I also characterized which of these groups are isomorphic to the complex braid group $B(e,e,n)$ and got a complete classification. 

\subsubsection{Presentations of Shi}

Shi \cite{ShiGenericVersion} discovered that the Artin group $B(W)$ for $W$ an affine Coxeter group of type $\tilde{A}_{n-1}$ is isomorphic to a group defined by a presentation obtained from the presentation of BMR of $B(e,e,n)$ by removing the relation between $\tilde{t}_0$ and $\tilde{t}_1$. This group is denoted by $B(\infty,\infty,n)$.

\begin{proposition}\label{PropPresShiB}
The group $B(\infty,\infty,n)$ is defined by a presentation with generators $\tilde{t}_0$, $\tilde{t}_1$, $\tilde{s}_3$, $\tilde{s}_4$, $\cdots$, $\tilde{s}_n$ and relations:
\begin{enumerate}

\item The braid relations for $\tilde{s}_3$, $\tilde{s}_4$, $\cdots$, $\tilde{s}_{n-1}$ given earlier in Definition \ref{DefinitionClasBraidGroup},
\item $\tilde{s}_3 \tilde{t}_i \tilde{s}_3 = \tilde{t}_i \tilde{s}_3 \tilde{t}_i$ for $i=0,1$,
\item $\tilde{s}_j \tilde{t}_i = \tilde{t}_i \tilde{s}_j$ for $i=0,1$ and $4 \leq j \leq n$,
\item $\tilde{s}_3 \tilde{t}_1\tilde{t}_0\tilde{s}_3\tilde{t}_1\tilde{t}_0 = \tilde{t}_1 \tilde{t}_0\tilde{s}_3\tilde{t}_1\tilde{t}_0\tilde{s}_3$

\end{enumerate}

\end{proposition}

This presentation can be described by the following diagram. It is called the presentation of Shi of $B(\infty,\infty,n)$

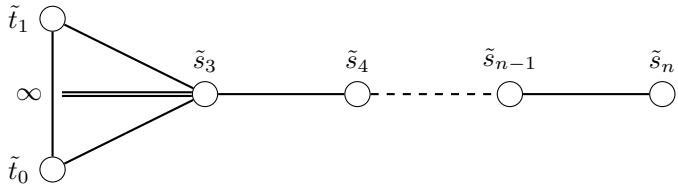
\begin{figure}[H]\label{PresofBMRBeen}
\begin{center}
\begin{tikzpicture}

\node[draw, shape=circle, label=left:$\tilde{t}_0$] (1) at (0,0) {};
\node[draw, shape=circle, label=left:$\tilde{t}_1$] (2) at (0,2) {};
\node[draw, shape=circle,label=above:$\tilde{s}_3$] (3) at (2,1) {};
\node[draw, shape=circle,label=above:$\tilde{s}_4$] (4) at (4,1) {};
\node[draw, shape=circle,label=above:$\tilde{s}_{n-1}$] (n-1) at (6,1) {};
\node[draw,shape=circle,label=above:$\tilde{s}_n$] (n) at (8,1) {};

\node[] (0) at (0,1) {};

\draw[thick,-] (1) to node[auto] {$\infty$} (2);
\draw[thick,-] (1) to (3);
\draw[thick,-] (2) to (3);
\draw[thick,-] (3) to (4);
\draw[thick,dashed,-] (4) to (n-1);
\draw[thick,-] (n-1) to (n);
\draw[double,thick,-] (3) to (0);

\end{tikzpicture}
\end{center}

\caption{Diagram for the presentation of Shi of $B(\infty,\infty,n)$.}
\end{figure}

It is also shown in \cite{ShiGenericVersion} that if one adds the quadratic relations for all the generators of the presentation of $B(\infty,\infty,n)$, one obtains a presentation of a group that we denote by $G(\infty,\infty,n)$. This group can be seen as a generic version of the group $G(e,e,n)$ where the root of unity $exp(2i\pi / e)$ is replaced by a parameter $x$.

\begin{definition}\label{DefG(infty,n)}

Let $x$ be a transcendental parameter. The group $G(\infty,\infty,n)$ is defined as the group of $n \times n$ monomial matrices, where
\begin{itemize}
\item all nonzero entries are of the form $x^k$ for $k \in \mathbb{Z}$ and
\item the product of all the nonzero entries is equal to 1.
\end{itemize} 

\end{definition}

Let $\mathbf{t}_0$, $\mathbf{t}_1$, $\mathbf{s}_3$, $\mathbf{s}_4$, $\cdots$, and $\mathbf{s}_n$ be the generators of the presentation of Shi of $G(\infty,\infty,n)$. The matrices in $G(\infty,\infty,n)$ that correspond to the generators are given by $\mathbf{t}_i \longmapsto t_i:= \begin{pmatrix}

0 & x^{-i} & 0\\
x^{i} & 0 & 0\\
0 & 0 & I_{n-2}\\

\end{pmatrix}$ for $i=0,1$ and $\mathbf{s}_j \longmapsto s_j:= \begin{pmatrix}

I_{j-2} & 0 & 0 & 0\\
0 & 0 & 1 & 0\\
0 & 1 & 0 & 0\\
0 & 0 & 0 & I_{n-j}\\

\end{pmatrix}$ for $3 \leq j \leq n$. In \cite{ShiGenericVersion}, Shi proved the following.

\begin{proposition}

The group $G(\infty,\infty,n)$ is isomorphic to the affine Coxeter group of type $\tilde{A}_{n-1}$.

\end{proposition}

\subsubsection{Presentations of Corran--Lee--Lee}

Corran--Lee--Lee \cite{CorranLeeLee} obtained a Garside structure for the group $B(\infty,\infty,n)$ by establishing another presentation of this group and constructing a suitable Garside structure. Similarly to the presentation of Corran--Picantin of $B(e,e,n)$ that consists in attaching a dual presentation of the dihedral group to standard presentations of type $A$, the presentation of Corran--Lee--Lee consists in attaching a dual presentation for the free group on 2 generators $\tilde{A}_1$ to standard presentations of type $A$.

\begin{definition}

The dual presentation of the free group $\tilde{A}_1$ is the following:

$$\tilde{A}_1 = < t_i, i \in \mathbb{Z}\ |\ t_it_{i-1} = t_jt_{j-1}\ \mathrm{for\ all}\ i,j \in \mathbb{Z}>.$$

\end{definition}

\begin{remark}

Note that this dual presentation of the free group on 2 generators gives rise to a dual Garside structure. The construction generalizes to the free groups on $n$ generators. This result is due to Bessis (see \cite{BessisDualFreeGroups}).

\end{remark}

The presentation of Corran--Lee--Lee of $B(\infty,\infty,n)$ is as follows.

\begin{proposition}\label{PropositionPresCLLB(infty,n)}
The group $B(\infty,\infty,n)$ is isomorphic to a group defined by the presentation with generators $\tilde{t}_i$ ($i \in \mathbb{Z}$), $\tilde{s}_3$, $\tilde{s}_4$, $\cdots$, $\tilde{s}_n$ and relations as follows.
\begin{enumerate}

\item the braid relations for $\tilde{s}_3$, $\tilde{s}_4$, $\cdots$, $\tilde{s}_{n-1}$ given earlier in Definition \ref{DefinitionClasBraidGroup},
\item $\tilde{s}_3 \tilde{t}_i \tilde{s}_3 = \tilde{t}_i \tilde{s}_3 \tilde{t}_i$ for $i \in \mathbb{Z}$,
\item $\tilde{s}_j \tilde{t}_i = \tilde{t}_i \tilde{s}_j$ for $i \in \mathbb{Z}$ and $4 \leq j \leq n$,
\item $\tilde{t}_i \tilde{t}_{i-1} = \tilde{t}_j \tilde{t}_{j-1}$ for $i,j \in \mathbb{Z}$.

\end{enumerate}

\end{proposition}

This presentation can be described by the diagram in Figure \ref{DiagramPresCorranLeeLeeBATilda}. Relation 4 is described by the vertical line such that the nodes $\mathbf{t}_i$'s are tangent to this line (this corresponds to the dual presentation of the free group $\tilde{A}_1$). The other edges used to describe all the other relations follow the standard conventions for the usual braid group diagram of type $A$.\\

\begin{figure}[]
\begin{center}
$$\begin{xy}
(1.8,-25) *++={\vdots};
(1.8, 25) *++={\vdots} **@{-};
(4, 16) *++={\rule{0pt}{4pt}} *\frm{o};
    (20,0) *++={\rule{0pt}{4pt}} *\frm{o} **@{-};
(4,-16) *++={\rule{0pt}{4pt}} *\frm{o};
    (20,0) *++={\rule{0pt}{4pt}} *\frm{o} **@{-};
(4, 8) *++={\rule{0pt}{4pt}} *\frm{o};
    (20,0) *++={\rule{0pt}{4pt}} *\frm{o} **@{-};
(4,-8) *++={\rule{0pt}{4pt}} *\frm{o};
    (20,0) *++={\rule{0pt}{4pt}} *\frm{o} **@{-};
(4, 0) *++={\rule{0pt}{4pt}} *\frm{o};
    (20,0) *++={\rule{0pt}{4pt}} *\frm{o} **@{-};
(30, 0) *++={\rule{0pt}{4pt}} *\frm{o} **@{-};
(40, 0) *++={\rule{0pt}{4pt}} *\frm{o} **@{-};
(50, 0) *++={\dots}  **@{-};
(60, 0) *++={\rule{0pt}{4pt}} *\frm{o} **@{-};
(0,-18) *++={};
(6,12) *++={\tilde{t}_2};
(6, 4) *++={\tilde{t}_1};
(6,-4) *++={\tilde{t}_0};
(7,-12)*++={\tilde{t}_{-1}};
(7,-20)*++={\tilde{t}_{-2}};
(23,-3) *++={\tilde{s}_3};
(33,-3) *++={\tilde{s}_4};
(43,-3) *++={\tilde{s}_5};
(63,-3) *++={\tilde{s}_n}
\end{xy}$$
\caption{Diagram for the presentation of Corran--Lee--Lee of $B(\infty,\infty,n)$.}\label{DiagramPresCorranLeeLeeBATilda}
\end{center}
\end{figure}
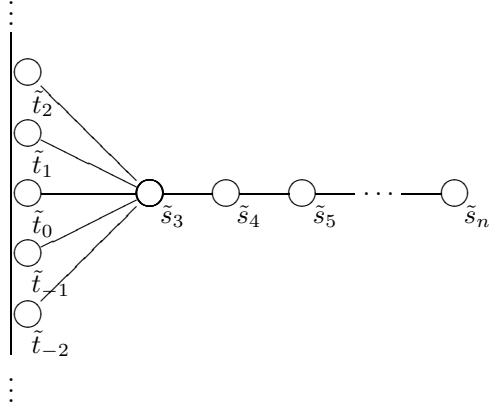

It is also shown in \cite{CorranLeeLee} that if one adds the quadratic relations for all the generators, one obtains a presentation of the group $G(\infty,\infty,n)$ that is isomorphic to the affine Coxeter group $\tilde{A}_{n-1}$. It is called the presentation of Corran--Lee--Lee of $G(\infty,\infty,n)$. Its diagram is similar to the diagram of the presentation of Corran--Lee--Lee of $B(\infty,\infty,n)$.\\

Denote by $\mathbf{X}$ the set $\{ \mathbf{t}_i\ ( i \in \mathbb{Z}), \mathbf{s}_3, \mathbf{s}_4, \cdots, \mathbf{s}_n \}$ of the generators of the presentation of Corran--Lee--Lee of $G(\infty,\infty,n)$. The matrices in $G(\infty,\infty,n)$ that correspond to the generators are given by $\mathbf{t}_i \longmapsto t_i:= \begin{pmatrix}

0 & x^{-i} & 0\\
x^{i} & 0 & 0\\
0 & 0 & I_{n-2}\\

\end{pmatrix}$ \mbox{for $i \in \mathbb{Z}$} and $\mathbf{s}_j \longmapsto s_j:= \begin{pmatrix}

I_{j-2} & 0 & 0 & 0\\
0 & 0 & 1 & 0\\
0 & 1 & 0 & 0\\
0 & 0 & 0 & I_{n-j}\\

\end{pmatrix}$ for $3 \leq j \leq n$. Denote by $X$ the set $\{t_i\ ( i \in \mathbb{Z}), s_3, \cdots, s_n\}$.\\

To avoid confusion, we use regular letters for matrices in $G(\infty,\infty,n)$ and bold letters for words over $\mathbf{X}$. We also set the following convention.

\begin{convention}\label{ConventionDecreaseIncreaseIndex}

A decreasing-index expression of the form $\mathbf{s}_i \mathbf{s}_{i-1} \cdots \mathbf{s}_{i'}$ is the empty word when $i < i'$ and an increasing-index expression of the form $\mathbf{s}_i \mathbf{s}_{i+1} \cdots \mathbf{s}_{i'}$ is the empty word when $i > i'$. Similarly, in $G(\infty,\infty,n)$, a decreasing-index product of the form $s_i s_{i-1} \cdots s_{i'}$ is equal to $I_n$ when $i < i'$ and an increasing-index product of the form $s_i s_{i+1} \cdots s_{i'}$ is equal to $I_n$ when $i > i'$, where $I_n$ is the identity $n \times n$ matrix.

\end{convention}

\section{Construction of the interval Garside structures}

In \cite{GeorgesNeaimeIntervals}, I described interval Garside structures for the complex braid group $B(e,e,n)$ associated to the complex reflection group $G(e,e,n)$. This section is an adaptation of this proof to the case of the group $B(\infty,\infty,n)$ that is isomorphic to the affine Artin group of type $\tilde{A}_{n-1}$. We therefore establish interval Garside structures for $B(\infty,\infty,n)$. We also obtain other interval Garside structures that appear naturally in our construction. As in \cite{GeorgesNeaimeIntervals}, the first step is to define a combinatorial technique in order to describe minimal word decompositions of the elements in $G(\infty,\infty,n)$ over the generating set $\mathbf{X}$ of Corran--Lee--Lee introduced in the previous section. This will enable us to describe elements that are of maximal length and to show which of these elements are balanced. Using lcm (least common multiples) computation, we are able to prove that the intervals of the balanced elements are lattices. This gives automatically rise to interval Garside structures by applying Theorem \ref{TheoremMichelGarside}. Next, using a property that is similar to Matsumoto's property for finite Coxeter groups, we identify Garside groups that are isomorphic to $B(\infty,\infty,n)$. We also obtain a Garside structure for the complex braid group $B(de,e,n)$ that is similar to the one discovered by Corran--Lee--Lee in \cite{CorranLeeLee}. The other Garside groups that appear in our construction are described by amalgamated products of copies of $B(\infty,\infty,n)$ over a common subgroup that is a finite-type Artin group of type $B_{n-1}$. In this section, we follow the proof in \cite{GeorgesNeaimeIntervals}. We state and develop the main results in each part of the proof that we adapt to the case $(\infty,\infty,n)$. We also include examples to clarify the statements. We do not include the proofs that are essentially similar to the proofs \mbox{in \cite{GeorgesNeaimeIntervals}.}

\subsection{Geodesic normal forms for $G(\infty,\infty,n)$}

Let us start by describing geodesic normal forms for $G(\infty,\infty,n)$. Actually, we define an algorithm that produces a word representative for each element of $G(\infty,\infty,n)$ over $\mathbf{X}$, where $\mathbf{X}$ is the set of the generators of the presentation of Corran--Lee-Lee. Then we prove that these word representatives are geodesic. Hence we get geodesic normal forms for $G(\infty,\infty,n)$. As an application, we determine the elements of $G(\infty,\infty,n)$ that are of maximal length over $\mathbf{X}$.\\

We introduce Algorithm \ref{Algo1} below that produces a word $R\!E(w)$ over $\mathbf{X}$ for a given matrix $w$ in $G(\infty,\infty,n)$. Note that we use Convention \ref{ConventionDecreaseIncreaseIndex} in the elaboration of the algorithm. The expression $R\!E(w)$ turns out to be a reduced expression over $\mathbf{X}$ of $w$ (see Proposition \ref{PropREwRedExp}).\\

Let $w_n := w \in G(\infty,\infty,n)$. For $i$ from $n$ to $2$, the $i$-th step of the algorithm transforms the block diagonal matrix $\left(
\begin{array}{c|c}
w_i & 0 \\
\hline
0 & I_{n-i}
\end{array}
\right)$ into a block diagonal matrix $\left(
\begin{array}{c|c}
w_{i-1} & 0 \\
\hline
0 & I_{n-i+1}
\end{array}
\right) \in G(\infty,\infty,n)$ with $w_1 = 1$. Actually, for $2 \leq i \leq n$, there exists a unique $c$ with $1 \leq c \leq n$ such that $w_i[i,c] \neq 0$. At each step $i$ of the algorithm, if $w_i[i,c] =1$, we shift it into the diagonal position $[i,i]$ by right multiplication by transpositions of the symmetric group $\mathfrak S_n$. If $w_i[i,c] \neq 1$, we shift it into the first column by right multiplication by transpositions, transform it into $1$ by right multiplication by an element of $\{t_i \ |\ i \in \mathbb{Z} \}$, and then shift the $1$ obtained into the diagonal position $[i,i]$.\\

We provide an example in order to better understand the algorithm. At each step, we indicate the values of $i$, $k$, and $c$ such that $w_i[i,c] = x^k$.

\begin{example}\label{examp algo NF}

We apply the algorithm to $w := \begin{pmatrix}

0 & 0 & 0 & 1\\
0 & x^{-1} & 0 & 0\\
0 & 0 & x^2 & 0\\
x^{-1} & 0 & 0 & 0\\

\end{pmatrix}$ $\in G(\infty,\infty,4)$.\\
Step $1$ $(i=4, k=-1, c=1)$: $w':=wt_{-1} = \begin{pmatrix}

0 & 0 & 0 & 1\\
x^{-2} & 0 & 0 & 0\\
0 & 0 & x^2 & 0\\
0 & 1 & 0 & 0\\

\end{pmatrix}$, then $w' := w's_3s_4 = \begin{pmatrix}

0 & 0 & 1 & 0\\
x^{-2} & 0 & 0 & 0\\
0 & \boxed{x^2} & 0 & 0\\
0 & 0 & 0 & \mathbf{1}\\

\end{pmatrix}$.\\
Step $2$ $(i=3, k=2, c=2)$: $w' := w's_2 = \begin{pmatrix}

0 & 0 & 1 & 0\\
0 & x^{-2} & 0 & 0\\
x^2 & 0 & 0 & 0\\
0 & 0 & 0 & 1\\

\end{pmatrix}$,\\ then $w' := w't_2 = \begin{pmatrix}

0 & 0 & 1 & 0\\
1 & 0 & 0 & 0\\
0 & 1 & 0 & 0\\
0 & 0 & 0 & 1\\

\end{pmatrix}$, then $w':=w's_3 = \begin{pmatrix}

0 & 1 & 0 & 0\\
\boxed{1} & 0 & 0 & 0\\
0 & 0 & \mathbf{1} & 0\\
0 & 0 & 0 & 1\\

\end{pmatrix}$.\\
Step $3$ $(i=2, k=0, c=1)$: $w' := w's_2 = I_{4}$. \vspace{0.2cm} Hence $R\!E(w) = \mathbf{s}_2 \mathbf{s}_3 \mathbf{t}_2 \mathbf{s}_2 \mathbf{s}_4 \mathbf{s}_3 \mathbf{t}_{-1}$. Recall that $\mathbf{s_2} = \mathbf{t}_0$. Thus, $R\!E(w) = \mathbf{t}_0 \mathbf{s}_3 \mathbf{t}_2 \mathbf{t}_0 \mathbf{s}_4 \mathbf{s}_3 \mathbf{t}_{-1}$.\\

\end{example}

\begin{algorithm}[H]\label{Algo1}

\SetKwInOut{Input}{Input}\SetKwInOut{Output}{Output}

\noindent\rule{12cm}{0.5pt}

\Input{$w$, a matrix in $G(\infty,\infty,n)$, with $n \geq 2$.}
\Output{$R\!E(w)$, a word over $\mathbf{X}$.}

\noindent\rule{12cm}{0.5pt}

\textbf{Local variables}: $w'$, $R\!E(w)$,  $i$, $c$, $k$.

\noindent\rule{12cm}{0.5pt}

\textbf{Initialisation}:
$s_2 := t_0$, $\mathbf{s}_2 := \mathbf{t}_0$, $R\!E(w) := \varepsilon$ (the empty word), $w' := w$.

\noindent\rule{12cm}{0.5pt}

\For{$i$ $\mathrm{\mathbf{from}}$ $n$ $\mathrm{\mathbf{down\ to}}$ $2$} {
	$c:=1$; $k:=0$; \\
	\While{$w'[i,c] = 0$}{$c :=c+1$\; 
	}
	 \textit{\#Then $w'[i,c]$ is the nonzero entry on the row $i$}\;
	 
\textit{\#Suppose $w'[i,c] = x^k$ for $k \in \mathbb{Z}$}\;

		\If{$k \neq 0$}{
		$w' := w's_{c}s_{c-1} \cdots s_{3}s_{2}t_{k}$; \textit{\#Then $w'[i,2] =1$}\;
		$R\!E(w) := \mathbf{t}_k\mathbf{s}_2\mathbf{s}_3 \cdots \mathbf{s}_c R\!E(w)$\;
		$c:=2$\;
	}	
	$w' := w's_{c+1} \cdots s_{i-1} s_{i}$; \textit{\#Then $w'[i,i] = 1$}\;
	$R\!E(w) := \mathbf{s}_i \mathbf{s}_{i-1} \cdots \mathbf{s}_{c+1} R\!E(w)$\;
}
\textbf{return} $R\!E(w)$;

\noindent\rule{12cm}{0.5pt}

\caption{A word over $\mathbf{X}$ corresponding to an element $w \in G(\infty,\infty,n)$.}
\end{algorithm}

\vspace{0.5cm}

The next lemma follows directly from the algorithm. 

\begin{lemma}\label{LemmaBlocks}

For $2 \leq i \leq n$, let $c$ denote the index with $w_i[i,c] \neq 0$. The block $w_{i-1}$ is obtained by
\begin{itemize}

\item removing the row $i$ and the column $c$ from $w_i$, then by

\item multiplying the first column of the new matrix by $w_i[i,c]$.

\end{itemize}

\end{lemma}

\begin{example}

Let $w$ be as in Example \ref{examp algo NF}, where $n = 4$. We apply Lemma \ref{LemmaBlocks}. The block $w_{3}$ is obtained by removing the row number $4$ and first column from $w_4 = w$ to obtain $\begin{pmatrix}

0 & 0 & 1\\
x^{-1} & 0 & 0\\
0 & x^2 & 0\\

\end{pmatrix}$, then by multiplying the first column of this matrix by $w[4,1] = x^{-1}$. The same can be said for the other block $w_2$.

\end{example}

\begin{definition}\label{DefREiw}

Let $2 \leq i \leq n$. Denote by $w_i[i,c]$ the unique nonzero entry on the row $i$ with $1 \leq c \leq i$.

\begin{itemize}

\item If $w_{i}[i,c] =1$, we define $R\!E_{i}(w)$ to be the word \\ 
$\mathbf{s}_i \mathbf{s}_{i-1} \cdots \mathbf{s}_{c+1}$ (decreasing-index expression).
\item If $w_{i}[i,c] = x^{k}$ with $k \neq 0$, we define $R\!E_{i}(w)$ to be the word \\
			\begin{tabular}{ll}
			$\mathbf{s}_i \cdots \mathbf{s}_3 \mathbf{t}_k$ & if $c=1$,\\
			$\mathbf{s}_i \cdots \mathbf{s}_3 \mathbf{t}_k \mathbf{t}_0$  & if $c=2$,\\
			$\mathbf{s}_i \cdots \mathbf{s}_3 \mathbf{t}_k \mathbf{t}_0 \mathbf{s}_3 \cdots \mathbf{s}_c$ & if $c \geq 3$.
			\end{tabular}

\end{itemize}

\end{definition}

\begin{lemma}

We have $R\!E(w) = R\!E_2(w) R\!E_3(w) \cdots R\!E_n(w)$.

\end{lemma}

\begin{proof}

The output $R\!E(w)$ of the algorithm is a concatenation of the words $R\!E_2(w),\\ R\!E_3(w), \cdots$, and $R\!E_n(w)$ obtained at each step $i$ from $n$ to $2$ of the algorithm.

\end{proof}

\begin{example}

If $w$ is defined as in Example \ref{examp algo NF}, we have
$$R\!E(w) = \underset{R\!E_{2}(w)}{\underbrace{\mathbf{t}_0}} \hspace{0.2cm} \underset{R\!E_{3}(w)}{\underbrace{\mathbf{s}_3\mathbf{t}_2\mathbf{t}_0}} \hspace{0.2cm} \underset{R\!E_{4}(w)}{\underbrace{ \mathbf{s}_4\mathbf{s}_3\mathbf{t}_{-1}}}.$$

\end{example}

\begin{proposition}

The output $R\!E(w)$ of the algorithm is a word representative over $\mathbf{X}$ of $w \in G(\infty,\infty,n)$.

\end{proposition}

\begin{proof}

The algorithm transforms the matrix $w$ into $I_n$ by multiplying it on the right by elements of $X$. We get $wx_1 \cdots x_r = I_n$, where $x_1$, $\cdots$, $x_r$ are elements of $X$. Hence $w = x_r^{-1} \cdots x_1^{-1} = x_r \cdots x_1$ since $x_i^2 = 1$ holds for all $x_i \in X$. The output of the algorithm is $R\!E(w) = \mathbf{x}_r \cdots \mathbf{x}_1$. Hence it is a word representative over $\mathbf{X}$ of $w \in G(\infty,\infty,n)$.

\end{proof}

As in \cite{GeorgesNeaimeIntervals} (Propositions 3.12 to 3.13), we obtain the following statements.

\begin{proposition}\label{PropREwRedExp}

Let $w$ be an element of $G(\infty,\infty,n)$. The word $R\!E(w)$ is a reduced expression over $\mathbf{X}$ of $w$.

\end{proposition}

\begin{proposition}\label{PropLengthdecreas}

Let $w$ be an element of $G(\infty,\infty,n)$. Denote by $a_i$ the unique nonzero entry $w[i,c_i]$ on the row $i$ of $w$ where $1 \leq i, c_i \leq n$.

\begin{enumerate}

\item For $3 \leq i \leq n$, we have:
	\begin{enumerate}
		\item if $c_{i-1} < c_i$, then
				\begin{center}$\ell(s_{i}w) = \ell(w)-1$ if and only if $a_{i} \neq 1,$\end{center}
		\item if $c_{i-1} > c_i$, then
				\begin{center}$\ell(s_{i}w) = \ell(w)-1$ if and only if $a_{i-1} =1.$\end{center}
	\end{enumerate}
\item If $c_1 < c_2$, then for all $k \in \mathbb{Z}$, we have
				\begin{center}$\ell(t_k w) = \ell(w)-1$ if and only if $a_2 \neq 1.$\end{center}
\item If $c_1 > c_2$, then for all $k \in \mathbb{Z}$, we have
				\begin{center}$\ell(t_k w) = \ell(w)-1$ if and only if $a_1 = x^{-k}.$\end{center}
				
\end{enumerate}

\end{proposition}

Using the geodesic normal forms of the elements of $G(\infty,\infty,n)$ built by the algorithm, we characterize the elements that are of maximal length over the generating set of the presentation of Corran--Lee--Lee of $G(\infty,\infty,n)$.

\begin{proposition}\label{PropMaxLength}

Let $n \geq 2$. The maximal length of an element of $G(\infty,\infty,n)$ is $n(n-1)$. It is realized for diagonal matrices $w$ such that $w[i,i]$ is equal to $x^k$ for $k \in \mathbb{Z}^{*}$. A minimal word representative of such an element is of the form $$(\mathbf{t}_{k_2}\mathbf{t}_0)(\mathbf{s}_3\mathbf{t}_{k_3}\mathbf{t}_0\mathbf{s}_3)\cdots (\mathbf{s}_n \cdots \mathbf{s}_3\mathbf{t}_{k_n}\mathbf{t}_0\mathbf{s}_3 \cdots \mathbf{s}_n),$$ with $k_2, \cdots, k_n \in \mathbb{Z}^{*}$.

\end{proposition}

\begin{proof}

By the algorithm, an element $w$ in $G(\infty,\infty,n)$ is of maximal length when $w_i[i,i] = x^k$ for $2 \leq i \leq n$ and $k \neq 0$. By Lemma \ref{LemmaBlocks}, this condition is satisfied when $w$ is a diagonal matrix such that $w[i,i]$ is equal to $x^k$ for $k \in \mathbb{Z}^{*}$. A minimal word representative given by the algorithm for such an element is of the form $(\mathbf{t}_{k_2}\mathbf{t}_0)(\mathbf{s}_3\mathbf{t}_{k_3}\mathbf{t}_0\mathbf{s}_3)\cdots (\mathbf{s}_n \cdots \mathbf{s}_3\mathbf{t}_{k_n}\mathbf{t}_0\mathbf{s}_3 \cdots \mathbf{s}_n)$ with $k_2, \cdots, k_n \in \mathbb{Z}^{*}$ which is of length $n(n-1)$.

\end{proof}

\begin{definition}\label{DefinitionLambda}

Denote by $\lambda$ the element $\begin{pmatrix}

(x^{-1})^{(n-1)} & & & \\
 & x & & \\
 & & \ddots & \\
 & & & x\\

\end{pmatrix} \in G(\infty,\infty,n)$. We have $R\!E(\lambda) = (\mathbf{t}_{1}\mathbf{t}_0)(\mathbf{s}_3\mathbf{t}_{1}\mathbf{t}_0\mathbf{s}_3)\cdots (\mathbf{s}_n \cdots \mathbf{s}_3\mathbf{t}_{1}\mathbf{t}_0\mathbf{s}_3 \cdots \mathbf{s}_n)$. Hence ${\ell}(\lambda)$ is equal to $n(n-1)$ which is the maximal length of an element of $G(\infty,\infty,n)$.

\end{definition}

\subsection{Balanced elements of maximal length}

The explicit geodesic normal forms constructed in the previous section enable us to characterize the interval $[1,\lambda^k]$ of the divisors of $\lambda^k$ (see Definition \ref{DefinitionLambda}) for $k \in \mathbb{Z}^{*}$. Following the ideas of Section 4 in \cite{GeorgesNeaimeIntervals}, we are able to recognize whether an element $w \in G(\infty,\infty,n)$ belongs to the set $[1,\lambda^k]$ directly from its matrix form. For this purpose, let us introduce some nice combinatorial tools defined as follows.

\begin{definition}\label{DefinitionBullets}

An index $[i,c]$ is said to be a \emph{bubble} if $w[j,d] = 0$ for all $[j,d] \in \left\{ [j,d]\ \vert\ j \leq i\ and\ d\leq c  \right\} \setminus \left\{[i,c] \right\} $. When $[i,c]$ is a bubble, $w[i,c]$ is represented by an encircled element.

\end{definition}

\begin{definition}\label{DefinitionZwZ'w}
We define two sets of matrix indices $Z(w)$ and $Z'(w)$ as follows.
\begin{itemize}

\item $Z(w) := \left\{ [j,d]\ \vert\ j \leq i\ and\ d \leq c\ for\ some\ bubble\ [i,c] \right\}$.
\item $Z'(w)$ is the set of matrix indices not in $Z(w)$.

\end{itemize}

\end{definition}

We draw a path in the matrix $w$ that separates it into two parts such that the upper left-hand side is $Z(w)$ and the other side is $Z'(w)$.

\begin{example}

Let $w =
\left(
\begin{BMAT}{ccccc}{ccccc}
0 & 0 & 0  & \encircle{$x^2$} & 0 \\
0 & 0 & 0 & 0 & x \\
0 & 0 & \begin{small}\encircle{$x^{-1}$}\end{small} & 0 & 0 \\
\encircle{$1$} & 0 & 0 & 0 & 0 \\
0 & x^{-2} & 0 & 0 & 0
\addpath{(0,1,0)rurruurur}
\end{BMAT}
\right) \in G(\infty,\infty,5).
$ When $[i,c]$ is a bubble, $w[i,c]$ is an encircled element and the drawn path separates $Z(w)$ from $Z'(w)$.

\end{example}

The proof of the following Proposition is similar to the proof of Proposition 4.12 in \cite{GeorgesNeaimeIntervals}.

\begin{proposition}\label{PropZ'=1orZetae}

Let $w \in G(\infty,\infty,n)$. We have that $w \preceq \lambda^k$ if and only if, for all $[j,d] \in Z'(w)$, $w[j,d]$ is either $0$, $1$, or $x^k$.

\end{proposition}

\begin{example}

Let $w =
\left(
\begin{BMAT}{ccccc}{ccccc}
0 & 0 & 0  & \encircle{$1$} & 0 \\
0 & 0 & 0 & 0 & \boxed{x} \\
0 & 0 & \encircle{$1$} & 0 & 0 \\
\encircle{$x^{-1}$} & 0 & 0 & 0 & 0 \\
0 & \boxed{1} & 0 & 0 & 0
\addpath{(0,1,0)rurruurur}
\end{BMAT}
\right) \in G(\infty,\infty,5)$.\\ For all $[i,c] \in Z'(w)$, $w[i,c]$ is equal to $1$ or $x$ (these are the boxed entries of $w$). It follows immediately that $w \in [1,\lambda]$ holds.

\end{example}

\begin{example}

Let $w =
\left(
\begin{BMAT}{ccccc}{ccccc}
0 & 0 & 0  & \encircle{$1$} & 0 \\
0 & 0 & 0 & 0 & \boxed{1} \\
0 & 0 & \encircle{$x^{-2}$} & 0 & 0 \\
\encircle{$1$} & 0 & 0 & 0 & 0 \\
0 & \boxed{x^2} & 0 & 0 & 0
\addpath{(0,1,0)rurruurur}
\end{BMAT}
\right) \in G(\infty,\infty,5).
$\\ For all $[i,c] \in Z'(w)$, $w[i,c]$ is equal to $1$ or $x^2$ (these are the boxed entries of $w$). It follows immediately that $w \in [1,\lambda^2]$ holds.

\end{example}

\begin{example}

Let $w =
\left(
\begin{BMAT}{cccc}{cccc}
\encircle{$x^{-4}$} & 0 & 0  & 0 \\
0 & 0 & x & 0 \\
0 & x & 0 & 0 \\
0 & 0 & 0 & \boxed{x^2} 
\addpath{(0,3,0)rurrr}
\end{BMAT}
\right) \in G(\infty,\infty,4).$\\ There exists $[i,c] \in Z'(w)$ such that $w[i,c] = x^2$ (the boxed element in $w$). It follows immediately that $w \notin [1,\lambda]$. Moreover, there exists $[i',c'] \in Z'(w)$ such that $w[i',c'] = x$. Hence we have $w \notin [1,\lambda^2]$. 

\end{example}

The previous description of the interval $[1,\lambda^k]$ allows us to easily characterize the balanced elements that are of maximal length in $G(\infty,\infty,n)$. The proof of the next theorem is similar to the proof of Theorem 4.22 in \cite{GeorgesNeaimeIntervals}.

\begin{theorem}\label{PropAllBalancedofMaxLength}

The balanced elements of $G(\infty,\infty,n)$  that are of maximal length are precisely $\lambda^k$ with $k \in \mathbb{Z}^{*}$. The intervals $[1,\lambda^k]$ are characterized in Proposition \ref{PropZ'=1orZetae}.

\end{theorem}

\subsection{The lattice property and interval structures}

We construct the monoid $M([1,\lambda^k])$ associated with each of the intervals $[1,\lambda^k]$ with $k \in \mathbb{Z}^{*}$. By Theorem \ref{PropAllBalancedofMaxLength}, $\lambda^k$ is balanced. Hence, by Theorem \ref{TheoremMichelGarside}, in order to prove that $M([1,\lambda^k])$ is a Garside monoid, it remains to show that both posets $([1,\lambda^k],\preceq)$ and $([1,\lambda^k],\preceq_r)$ are lattices. In the next two propositions, we provide the least commonon multiples of $x$ and $y$ in $([1, \lambda^k], \preceq)$ and $([1, \lambda^k], \preceq_r)$ for $x,y \in X = \{ t_i, s_j\ |\ i \in \mathbb{Z}, 3 \leq j \leq n \}$ (see the proof of Propositions 5.5 and 5.6 \mbox{in \cite{GeorgesNeaimeIntervals}).}

\begin{proposition}\label{PropLCMofGenIn1lambdak}

Let $x, y \in X$. The least common multiple in $([1, \lambda^k], \preceq)$ of $x$ and $y$, denoted by $x \vee y$, exists and is given by the following identities:

\begin{itemize}

\item $t_i \vee t_j = t_kt_0 = t_{i}t_{i-k} = t_{j}t_{j-k}$ for $i \neq j \in \mathbb{Z}$,
\item $t_i \vee s_3 = t_is_3t_i = s_3t_is_3$ for $i \in \mathbb{Z}$,
\item $t_i \vee s_j = t_is_j = s_jt_i$ for $i \in \mathbb{Z}$ and $4 \leq j \leq n$,
\item $s_i \vee s_{i+1} = s_{i}s_{i+1}s_{i} = s_{i+1}s_{i}s_{i+1}$ for $3 \leq i \leq n-1$,
\item $s_i \vee s_j = s_is_j = s_js_i$ for $3 \leq i \neq j \leq n$ and $|i-j| > 1$.

\end{itemize}

\end{proposition}

\begin{proposition}\label{PropLCMofGenRight}

Let $x, y \in X$. The least common multiple in $([1, \lambda^k], \preceq_r)$ of $x$ and $y$, denoted by $x \vee_r y$, exists and is equal to $x \vee y$.

\end{proposition}

As in Corollary 5.13 of \cite{GeorgesNeaimeIntervals}, we immediately get the following.

\begin{theorem}\label{CorBothPosetsLattices}

Both posets $([1,\lambda^k],\preceq)$ and $([1,\lambda^k],\preceq_r)$ are lattices. 

\end{theorem}

We are ready to define the interval Garside monoid $M([1,\lambda^k])$.

\begin{definition}\label{DefIntervalMonoid}

Let $\underline{[1,\lambda^k]}$ be a set in bijection with $[1,\lambda^k]$ with $$[1,\lambda^k] \longrightarrow \underline{[1,\lambda^k]}: w \longmapsto \underline{w}.$$ We define the monoid $M([1,\lambda^k])$ by the following presentation of monoid with

\begin{itemize}

\item generating set: $\underline{[1,\lambda^k]}$,
\item relations: $\underline{w} = \underline{w'} \hspace{0.1cm} \underline{w''}$ for $w, w', w'' \in [1,\lambda^k]$, $w = w' w''$, and $\ell(w) = \ell(w') + \ell(w'')$.

\end{itemize}

\end{definition}

We have that $\lambda^k$ is balanced. Also, by Theorem \ref{CorBothPosetsLattices}, both posets $([1,\lambda^k],\preceq)$ and $([1,\lambda^k],\preceq_r)$ are lattices. Hence, by Theorem \ref{TheoremMichelGarside}, we have:

\begin{theorem}\label{TheoremIntervalStructure}

$(M([1,\lambda^k]),\underline{\lambda^k})$ is an interval Garside monoid with simples $\underline{[1,\lambda^k]}$. Its group of fractions exists and is denoted by $G([1,\lambda^k])$.

\end{theorem}

\section{Identifying Artin groups and other new groups}
 
Our first aim is to give a simple presentation of the interval Garside monoid $M([1,\lambda^k])$ for $k \in \mathbb{Z}^{*}$. 

\begin{definition}\label{DefofB+keen}

For $k \in \mathbb{Z}^{*}$, we define the monoid $B^{\oplus k}(\infty,\infty,n)$ by a monoid presentation with

\begin{itemize}

\item generating set: $\widetilde{X} = \{ \tilde{t}_{i}, \tilde{s}_{j} \ |\ i \in \mathbb{Z}, 3 \leq j \leq n \}$ and
\item relations: $\left\{
\begin{array}{ll}

\tilde{s}_{i}\tilde{s}_{j}\tilde{s}_{i} = \tilde{s}_{j}\tilde{s}_{i}\tilde{s}_{j} & for\ |i-j|=1,\\
\tilde{s}_{i}\tilde{s}_{j} = \tilde{s}_{j}\tilde{s}_{i} & for\ |i-j| > 1,\\
\tilde{s}_{3}\tilde{t}_{i}\tilde{s}_{3} = \tilde{t}_{i}\tilde{s}_{3}\tilde{t}_{i} & for\ i \in \mathbb{Z},\\
\tilde{s}_{j}\tilde{t}_{i} = \tilde{t}_{i}\tilde{s}_{j} & for\ i \in \mathbb{Z}\ and\ 4 \leq j \leq n,\\
\tilde{t}_{i}\tilde{t}_{i-k} = \tilde{t}_{j}\tilde{t}_{j-k} & for\ i, j \in \mathbb{Z}.

\end{array}
\right.$

\end{itemize}

\end{definition}

Note that the monoid $B^{\oplus 1}(\infty,\infty,n)$ corresponds to the monoid defined by the presentation of Corran--Lee--Lee considered as a monoid presentation (see Proposition~\ref{PropositionPresCLLB(infty,n)}). In \cite{CorranLeeLee}, this monoid is denoted by $B^{+}(\infty,\infty,n)$.\\

The proof of the next proposition is similar to the proof of Proposition 6.4 in \cite{GeorgesNeaimeIntervals}.

\begin{proposition}

The monoid $B^{\oplus k}(\infty,\infty,n)$ is isomorphic to $M([1,\lambda^k])$ for all $k \in \mathbb{Z}^{*}$.

\end{proposition}

Since $B^{\oplus k}(\infty,\infty,n)$ is isomorphic to $M([1,\lambda^k])$, we deduce that $B^{\oplus k}(\infty,\infty,n)$ is a Garside monoid and we denote by $B^{(k)}(\infty,\infty,n)$ its group of fractions.\\

The next proposition characterizes the monoids $B^{\oplus k}(\infty,\infty,n)$ that are isomorphic to the monoid $B^{+}(\infty,\infty,n)$. Its proof is similar to the proof of Proposition 6.6 in \cite{GeorgesNeaimeIntervals}. 

\begin{proposition}\label{PropIsomwithMonoidCP}

The monoids $B^{\oplus k}(\infty,\infty,n)$ and $B^{+}(e,e,n)$ are isomorphic if and only if $k = \pm 1$.

\end{proposition}

The next corollary is immediate. It identifies which of the groups of fractions of $B^{\oplus k}(\infty,\infty,n)$ are isomorphic to the group $B(\infty,\infty,n)$.

\begin{cor}\label{CorIsomGrFractBraidGroup}

The group $B^{(k)}(\infty,\infty,n)$ is isomorphic to the group $B(\infty,\infty,n)$ for $k = \pm 1$.

\end{cor}

Since $B(\infty,\infty,n)$ is isomorphic to the affine Artin group $B(\tilde{A}_{n-1})$ (see \cite{ShiGenericVersion}), the previous result states that $B(\tilde{A}_{n-1})$ is isomorphic to $B^{(k)}(\infty,\infty,n)$ when \mbox{$k = \pm 1$.} Therefore, the previous construction provides two interval Garside structures for $B(\tilde{A}_{n-1})$.\\

Recall that the finite-type Artin group $B(B_{n-1})$ is defined by a presentation with generators $q_1$, $q_2$, $\cdots$, $q_{n-1}$ and relations that can be described by the diagram presentation of Figure \ref{FigureB(Bn)}. By an adaptation of the results of Crisp \cite{CrispInjectivemaps} as in Lemma 5.2 of \cite{CalMarHomologyComputations}, we have the following embedding.

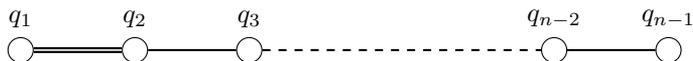
\begin{figure}[H]
\begin{center}
\begin{tikzpicture}

\node[draw, shape=circle, label=above:$q_1$] (1) at (0,0) {};
\node[draw, shape=circle, label=above:$q_2$] (2) at (1.5,0) {};
\node[draw, shape=circle, label=above:$q_3$] (3) at (3,0) {};
\node[draw, shape=circle, label=above:$q_{n-2}$] (n-2) at (7,0) {};
\node[draw, shape=circle, label=above:$q_{n-1}$] (n-1) at (8.5,0) {};

\draw[double,thick,-] (1) to (2);
\draw[thick,-] (2) to (3);
\draw[thick,-] (n-2) to (n-1);
\draw[thick,dashed] (3) to (n-2);

\end{tikzpicture}
\end{center}
\caption{Diagram presentation for the finite-type Artin group $B(B_{n-1})$.}\label{FigureB(Bn)}
\end{figure}

\begin{proposition}

The finite-type Artin group $B(B_{n-1})$ injects in $B^{(k)}(\infty,\infty,n)$.

\end{proposition}

\begin{proof}

Define a monoid homomorphism $\phi : B^{+}(B_{n-1}) \longrightarrow B^{\oplus k}(\infty,\infty,n) : q_{1} \longmapsto \tilde{t}_i\tilde{t}_{i-k}$, $q_2 \longmapsto \tilde{s}_3$, $\cdots$, $q_{n-1} \longmapsto \tilde{s}_n$. It is easy to check that for all $x, y \in \{q_1, q_2, \cdots, q_{n-1}\}$, we have $\phi(x) \vee \phi(y)=\phi(x \vee y)$. Hence by applying Lemma 5.2 of \cite{CalMarHomologyComputations}, the Artin group $B(B_{n-1})$ injects in $B^{(k)}(\infty,\infty,n)$.

\end{proof}

\begin{remark}

Consider the standard Garside element of the Artin group $B(B_{n-1})$ (see the standard construction explained in Subsection 2.4). It is of the form $$q_1(q_2q_1q_2) \cdots (q_{n-1}\cdots q_2 q_1 q_2 \cdots q_{n-1}).$$ It corresponds to the longest element in the associated finite Coxeter group of type $B_{n-1}$. The image of this Garside element by the injection of the previous proposition corresponds to the Garside element of $B^{\oplus k}(\infty,\infty,n)$ that is of the form $$\tilde{t}_1\tilde{t}_0(\tilde{s}_3\tilde{t}_1\tilde{t}_0\tilde{s}_3) \cdots (\tilde{s}_n\cdots \tilde{s}_3 \tilde{t}_1\tilde{t}_0 \tilde{s}_3 \cdots \tilde{s}_n).$$

\end{remark}

The complex braid group $B(de,e,n)$ ($d > 1$, $e \geq 1$, and $n > 1$) is described in \cite{CorranLeeLee} as the semi-direct product of an infinite cyclic group generated by $z$ and the group $B(\infty,\infty,n)$. On examination of the action of $z$ on $B^{+}(\infty,\infty,n) \simeq M([1,\lambda])$ given by the semi-direct product, one applies Theorem 4.1 of \cite{LeeSemiDirect} to deduce a Garside structure for $B(de,e,n)$. A similar result was obtained by Corran--Lee--Lee in \cite{CorranLeeLee}.\\

Assume now $k \neq \pm 1$, we describe $B^{(k)}(\infty,\infty,n)$ as an amalgamated product of $k$ copies of the group $B(\infty,\infty,n)$ over a common subgroup that is the finite-type Artin group $B(B_{n-1})$. We construct $B^{(k)}(\infty,\infty,n)$ as follows (see also Figure \ref{ConstructionBk(een)}).

\begin{proposition}\label{PropBkeenAmalgamatedProduct}

Let $B(1) := B(\infty,\infty,n)$. Inductively, define the amalgamated product $B(i+1) := B(i) *_{\begin{small}B(B_{n-1})\end{small}} B(\infty,\infty,n)$ over the Artin group $B(B_{n-1})$. Then $B^{(k)}(\infty,\infty,n)$ is isomorphic to $B(k)$.

\end{proposition}

\begin{proof}

Due to the presentation of $B^{(k)}(\infty,\infty,n)$ given in Definition \ref{DefofB+keen} and to the presentation of the amalgamated products (see Section 4.2 of \cite{CombinatorialGroupTheory}), one can deduce that $B(k)$ is isomorphic to $B^{(k)}(\infty,\infty,n)$.

\end{proof}

\begin{example}

Consider the case of $B^{(2)}(\infty,\infty,4)$. It is an amalgamated product of $2$ copies of $B(\infty,\infty,4)$ over the Artin group $B(B_3)$. Consider the diagram of this amalgamation on Figure \ref{AmalgamationExampleFigure}. The presentation of $B(\infty,\infty,4) *_{\begin{small}B(B_3)\end{small}} B(\infty,\infty,4)$ is as follows:
\begin{itemize}

\item the generators are the union of the generators of the two copies of $B(\infty,\infty,4)$, 
\item the relations are the union of the relations of the two copies of $B(\infty,\infty,4)$ with the additional relations $\tilde{t}_2\tilde{t}_0 = \tilde{t}_3\tilde{t}_1$, $\tilde{s}_3 = \tilde{s}'_3$, and $\tilde{s}_4 = \tilde{s}'_4$.

\end{itemize}
This is exactly the presentation of $B^{(2)}(\infty,\infty,4)$ given in Definition \ref{DefofB+keen}.
\end{example}

\begin{figure}[H]
\begin{center}
\begin{tikzpicture}

\node[draw, shape=rectangle, label=below:{$B(\infty,\infty,n)$}] (1) at (0,0) {};
\node[draw, shape=rectangle, label=below:{$B(\infty,\infty,n)$}] (2) at (2,0) {};
\node[draw, shape=rectangle, label=below:{$B(\infty,\infty,n)$}] (3) at (-1,1) {};
\node[draw, shape=rectangle, label=right:{$B(2)$}] (4) at (1,1) {};
\node[draw, shape=rectangle, label=right:{$B(3)$}] (5) at (0,2) {};
\node[draw, shape=rectangle, label=left:{$B(\infty,\infty,n)$}] (6) at (-3,3) {};
\node[draw, shape=rectangle] (7) at (-1,3) {};
\node[draw, shape=rectangle, label=right:{$B(k) \simeq B^{(k)}(\infty,\infty,n)$}] (8) at (-2,4) {};

\draw[thick,-] (1) to (4);
\draw[thick,-] (2) to (4);
\draw[thick,-] (5) to (3);
\draw[thick,-] (5) to (4);
\draw[thick,dashed,-] (7) to (5);
\draw[thick,dashed,-] (-1.5,1.5) to (-2.5,2.5);
\draw[thick,-] (8) to (6);
\draw[thick,-] (8) to (7);

\end{tikzpicture}
\end{center}
\caption{The construction of $B^{(k)}(\infty,\infty,n)$.}\label{ConstructionBk(een)}
\end{figure}
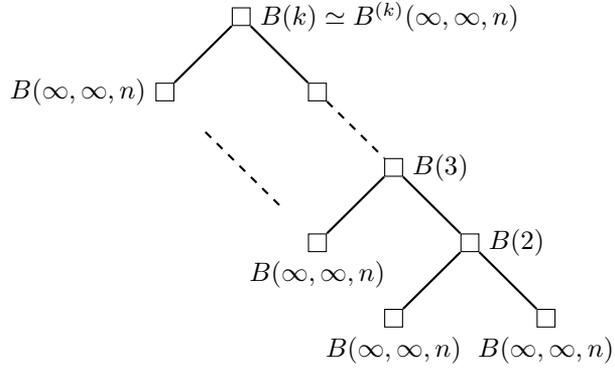

\begin{figure}[H]

\begin{center}
\begin{tikzpicture}
\node[draw, shape=circle, label=above:$\tilde{t}_3\tilde{t}_1$] (1) at (0,0) {};
\node[draw, shape=circle, label=above:$\tilde{s}_3$] (2) at (1,0) {};
\node[draw, shape=circle, label=above:$\tilde{s}_4$] (3) at (2,0) {};
\draw[double,thick,-] (1) to (2);
\draw[thick,-] (2) to (3);

\node at (4,0.1) {$\overset{\sim}{\longrightarrow}$};

\draw[thick,->] (1,-0.8) to (1,-1.6);

\node[draw, shape=circle, label=above:$\tilde{t}_2\tilde{t}_0$] (4) at (6,0) {};
\node[draw, shape=circle, label=above:$\tilde{s}'_3$] (5) at (7,0) {};
\node[draw, shape=circle, label=above:$\tilde{s}'_4$] (6) at (8,0) {};
\draw[double,thick,-] (4) to (5);
\draw[thick,-] (5) to (6);

\draw[thick,->] (7,-0.8) to (7,-1.6);

\end{tikzpicture}

\begin{tabular}{cccc}
\hspace{0.5cm} & $$\begin{xy}
(1.8,-25) *++={\vdots};
(1.8, 25) *++={\vdots} **@{-};
(4, 16) *++={\rule{0pt}{4pt}} *\frm{o};
    (20,0) *++={\rule{0pt}{4pt}} *\frm{o} **@{-};
(4,-16) *++={\rule{0pt}{4pt}} *\frm{o};
    (20,0) *++={\rule{0pt}{4pt}} *\frm{o} **@{-};
(4, 8) *++={\rule{0pt}{4pt}} *\frm{o};
    (20,0) *++={\rule{0pt}{4pt}} *\frm{o} **@{-};
(4,-8) *++={\rule{0pt}{4pt}} *\frm{o};
    (20,0) *++={\rule{0pt}{4pt}} *\frm{o} **@{-};
(4, 0) *++={\rule{0pt}{4pt}} *\frm{o};
    (20,0) *++={\rule{0pt}{4pt}} *\frm{o} **@{-};
(30, 0) *++={\rule{0pt}{4pt}} *\frm{o} **@{-};
(0,-18) *++={};
(6,12) *++={\tilde{t}_5};
(6, 4) *++={\tilde{t}_3};
(6,-4) *++={\tilde{t}_{1}};
(7,-12)*++={\tilde{t}_{-1}};
(7,-20)*++={\tilde{t}_{-3}};
(23,-3) *++={\tilde{s}_3};
(33,-3) *++={\tilde{s}_4};
\end{xy}$$
& \hspace{2cm} &  $$\begin{xy}
(1.8,-25) *++={\vdots};
(1.8, 25) *++={\vdots} **@{-};
(4, 16) *++={\rule{0pt}{4pt}} *\frm{o};
    (20,0) *++={\rule{0pt}{4pt}} *\frm{o} **@{-};
(4,-16) *++={\rule{0pt}{4pt}} *\frm{o};
    (20,0) *++={\rule{0pt}{4pt}} *\frm{o} **@{-};
(4, 8) *++={\rule{0pt}{4pt}} *\frm{o};
    (20,0) *++={\rule{0pt}{4pt}} *\frm{o} **@{-};
(4,-8) *++={\rule{0pt}{4pt}} *\frm{o};
    (20,0) *++={\rule{0pt}{4pt}} *\frm{o} **@{-};
(4, 0) *++={\rule{0pt}{4pt}} *\frm{o};
    (20,0) *++={\rule{0pt}{4pt}} *\frm{o} **@{-};
(30, 0) *++={\rule{0pt}{4pt}} *\frm{o} **@{-};
(0,-18) *++={};
(6,12) *++={\tilde{t}_4};
(6, 4) *++={\tilde{t}_2};
(6,-4) *++={\tilde{t}_0};
(7,-12)*++={\tilde{t}_{-2}};
(7,-20)*++={\tilde{t}_{-4}};
(23,-3) *++={\tilde{s}'_3};
(33,-3) *++={\tilde{s}'_4};
\end{xy}$$
\end{tabular}


\begin{tikzpicture}[scale=.8]

\node[] () at (0,0) {};
\draw[thick,->] (0,0) to (1,-1);
\draw[thick,->] (3.5,0) to (2.5,-1);

\end{tikzpicture}


$B^{(2)}(\infty,\infty,4)$

\end{center}
\caption{The amalgamation for $B^{(2)}(\infty,\infty,n)$.}\label{AmalgamationExampleFigure}
\end{figure}
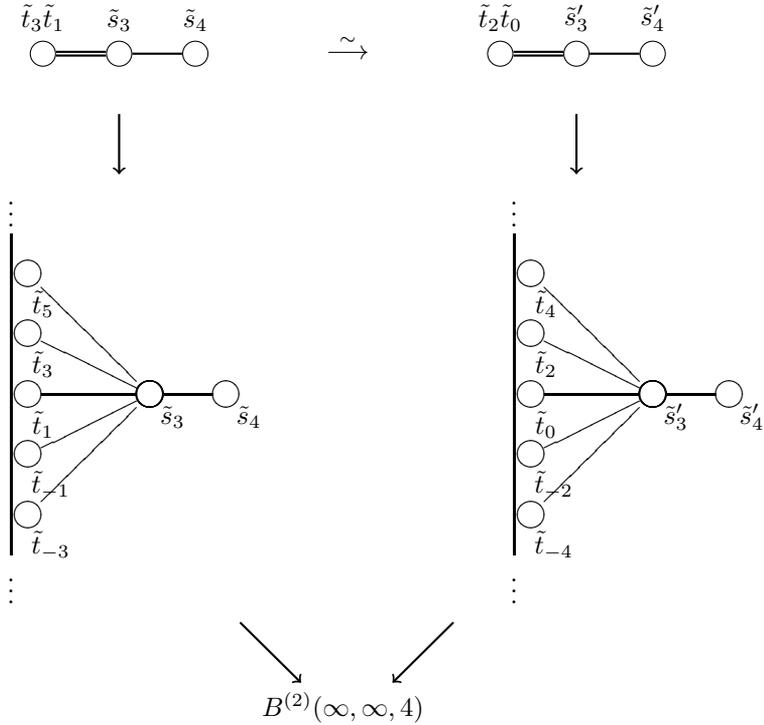

\bibliographystyle{plain}
\bibliography{arxiv-Garside-ATilda-V1-BIB}

\end{document}